\documentclass[11pt]{article}
\usepackage{latexsym}
\usepackage[margin=1in]{geometry}

\usepackage{graphicx,enumerate} 
\usepackage{color} 
\usepackage{amsmath, amsthm, amssymb}
\usepackage{enumerate}
\usepackage{float}
\usepackage{url}
\usepackage{stackrel}
\usepackage{mathrsfs,dsfont}
\usepackage{caption}
\usepackage{subcaption}
\usepackage{wrapfig}
\usepackage{bbm}
\usepackage{bm}
\usepackage{tikz}

\newtheorem{theorem}{Theorem}[section]
\newtheorem{corollary}[theorem]{Corollary}
\newtheorem{proposition}[theorem]{Proposition}
\newtheorem{lemma}[theorem]{Lemma}
\newtheorem{example}[theorem]{Example}

\newtheorem{conjecture}[theorem]{Conjecture}
\newtheorem{question}[theorem]{Question}

\theoremstyle{definition}
\newtheorem{definition}[theorem]{Definition}

\theoremstyle{remark}
\newtheorem{remark}[theorem]{Remark}

\newcommand{\St}{{S}}
\newcommand{\invtPoly}{\mathcal{P}}
\DeclareMathOperator{\im}{im}
\newcommand{\CRN}{chemical reaction network }

\newcommand{\R}{\mathbb{R}}
\newcommand{\Z}{\mathbb{Z}}
\newcommand{\Rnn}{\mathbb{R}_{\geq 0}}

\newcommand{\bd}{box diagram}

\def\lra{\leftrightarrows}
\def\rla{\rightleftarrows}
\def\been{\begin{enumerate}}
\def\enen{\end{enumerate}}

\def\SS{\mathcal S}
\def\CC{\mathcal C}
\def\RR{\mathcal R}

\def\wt{\widetilde}

\def\ep{\epsilon}

\def\caseOne{if $-1 <\alpha <0$}
\def\caseTwo{if $\alpha <-1$}

\def\ds{\displaystyle}

\def\la{\leftarrow}

\newcommand{\specialcell}[2][c]{\begin{tabular}[#1]{@{}l@{}}#2\end{tabular}}

\DeclareMathOperator{\capPSS}{cap_{pos}} 
\DeclareMathOperator{\capNPSS}{cap_{nondeg}}
\DeclareMathOperator{\capStable}{cap_{exp-stab}}

\newcommand{\lradot}{%
  \mathrel{\ooalign{\hfil$\vcenter{
   \hbox{$\mkern3mu\scriptscriptstyle\bullet$}}$\hfil\cr$\longleftrightarrow$\cr}
  }%
}

\begin{document}

\title{Which small reaction networks are multistationary?}

\date{\today}

\author{Badal Joshi and Anne Shiu}

\maketitle




\abstract{Reaction networks taken with mass-action kinetics arise in many settings, from epidemiology 
to population biology 
to systems of chemical reactions.    Bistable reaction networks are posited to underlie biochemical switches, which motivates the following question: which reaction networks have the capacity for multiple steady states?  
Mathematically, this asks: among certain parametrized families of polynomial systems, which admit multiple positive roots?  No complete answer is known. 
This work analyzes the smallest networks, those with only a few chemical species or reactions.
For these ``smallest'' networks, we completely answer the question of multistationarity and, in some cases, multistability too, thereby extending related work of Boros.  
Our results highlight the role played by the Newton polytope of a network (the convex hull of the reactant vectors).  
Also, our work is motivated by recent results that explain how 
a given network's capacity for multistationarity  
arises from that of certain related networks which are typically smaller.
Hence, we are interested in classifying small multistationary networks, and our work forms a first step in this direction.
}

\vskip .1in



\maketitle

\section{Introduction}
Many dynamical systems 
that arise in applications 
admit multiple steady states, and yet the following question is open:  
\begin{question} \label{q:main}
When taken with mass-action kinetics, 
which reaction networks have the capacity for two or more steady states?  
\end{question}
\noindent
Mathematically, Question~\ref{q:main} asks: among certain parametrized families of polynomial systems, which of them admit multiple positive solutions?  
 This question remains open, even as related questions have been resolved.  
For instance, Banaji and Pantea proved that, when taken with certain {\em power-law} kinetics, a network that contains all inflows/outflows has the capacity for two or more steady states if and only if its non-flow subnetwork is ``accordant''~\cite[Corollary 4.26]{BP}.  

Question~\ref{q:main} is the first step toward classifying bistable networks, which are thought to underlie biochemical switches. 
Much prior work has analyzed bistability in specific networks arising in systems biology (e.g.,~\cite{lipo,Enzyme-sharing,case-study,Hell-2015,HoHarrington,TG2,WangSontag}), but we would also like to develop theory that applies to general families of networks.  Here we succeed in obtaining such results.

Specifically, the aim of the present work is to completely answer Question~\ref{q:main} for small networks,  those with only 1 chemical species or up to 2 chemical reactions (each of which is possibly reversible).  Moreover, deciding multistationarity for these networks will be immediate from inspecting their ``reaction diagrams'' (which depict, for instance, the reaction $A \to B$ as an arrow from $(1,0)$ to $(0,1)$).  
Here we informally summarize some of our main results:
{\em 
Let $G$ be a reaction network with exactly $r$ reactions and $s$ species. Then:
\begin{enumerate}
	\item If $s=1$, then $G$ is nondegenerately multistationary if and only if $G$ has a subnetwork with ``arrow diagram'' $(\to, \la, \to)$ 
	(e.g., $0 \to A  \leftarrow 2A,~3A \to 4A$) or $(\la, \to, \la)$. 
	\item If $r=1$, then $G$ is not multistationary. If $r=2$, then $G$ is nondegenerately multistationary if and only if for some choice of species $i$ and $j$, the projection of the reaction diagram to the $(i,j)$-plane has one of the following zigzag forms (depicted also is the rectangle whose diagonal joins the two reactants):
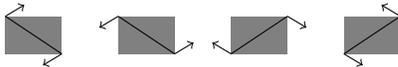
\begin{center}
	\begin{tikzpicture}[scale=.5]
	\path [fill=gray] (0,0) rectangle (1.5,1);
	\draw [->] (0,1) -- (.5,1.3);
	\draw [->] (1.5,0) -- (1,-.3);
	\draw (0,1) --(1.5,0); 
	\path [fill=gray] (3,0) rectangle (4.5,1);
	\draw [->] (3,1) -- (2.5,.7);
	\draw [->] (4.5,0) -- (5,.3);
	\draw (3,1) --(4.5,0); 
	\path [fill=gray] (6,0) rectangle (7.5,1);
	\draw [->] (6,0) -- (5.5,.3);
	\draw [->] (7.5,1) -- (8, .7);
	\draw (6,0) --(7.5,1); 
	\path [fill=gray] (9,0) rectangle (10.5,1);
	\draw [->] (9,0) -- (9.5, -.3);
	\draw [->] (10.5,1) -- (10,1.3);
	\draw (9,0) --(10.5,1);
	\end{tikzpicture}
\end{center}
and, if only one such pair $(i,j)$ exists, then the slope of the marked diagonal is {\em not} $-1$.
\end{enumerate}
}
\noindent
For a precise statement, 
see Theorems~\ref{thm:s=1-r=1} and~\ref{thm:speculation}.  Also, see Definition \ref{def:arrow-diagram} for ``arrow diagram".

In addition to taking a step toward answering Question~\ref{q:main}, our goal of cataloging small multistationary networks is motivated by recent results on how a given network's capacity for multistationarity arises from that of certain smaller networks~\cite{BP-inher,joshi2012atoms}.
Here is one such ``lifting'' result, stated informally:
if $N$ is a subnetwork of $G$ 
and both networks contain all possible inflow and outflow reactions, then if $N$ is multistationary then $G$ is too.  
Therefore, we would like a catalogue of small multistationary networks against which the networks $N$ can be checked.  

We began such a catalogue for those networks containing all inflows and outflows in~\cite{joshi2012atoms}, and 
then for general networks in~\cite{mss-review}.
The present work adds to this second list; our contribution is summarized in Theorem~\ref{thm:list-new-atoms}.  
Specifically, we find new {\em embedding-minimal} multistationary networks. ({\em Embedded networks} are obtained by removing species, as we saw earlier, and/or reactions; 
see 
Section~\ref{sec:atoms}). 

Although we are the first to propose a catalogue of small multistationary networks, there are some related analyses of small networks. 
Horn initiated the investigation of small chemical reaction networks by enumerating networks 
that consist of a few ``short complexes'' (at-most-bimolecular) \cite{HornCGph,Horn73}. 
Networks that consist of at most three short complexes do not admit multiple steady states.
Deckard~{\em et al.}\ enumerated networks (with certain restrictions) with few reactions and species~\cite{Deckard}.  Sampling from this enumeration, Pantea and Craciun~\cite{Pantea_comput} estimated the percentage of small networks (containing all inflows/outflows) for which the injectivity criterion of Craciun and Feinberg~\cite{ME1} precludes multistationarity.
Wilhelm and Heinrich found and analyzed the smallest networks admitting bistability or Hopf bifurcations~\cite{ThomasSmallest,smallestHopf,smallestHopf2}; see also~\cite{Smith}.
In earlier work, the present authors analyzed the smallest multistationary networks that are mass-preserving~\cite{Smallest}, and also those that contain all inflows/outflows plus one or two non-flow reactions involving at-most-bimolecular chemical complexes~\cite{joshi2013complete,joshi2012atoms}.  

Another related work is a recent paper of Boros~\cite{revisiting}, which our work uses and extends.  
Boros used Feinberg's deficiency one algorithm to classify small multistationary networks, more specifically, certain `regular' networks with only one species or up to two reactions (possibly reversible and allowing flow reactions as well).  
Boros considered only regular networks
because the deficiency one algorithm applies to those ones only.
Our work analyzes many of the same networks, {\em without} requiring them to be regular.  Moreover, like Boros, we elucidate the precise conditions for multistationarity, but also go beyond, to ascertain nondegenerate multistationarity and bistability as well.
For instance, we will see that networks with only two reactions and two species may admit up to two nondegenerate steady states, but only one is stable (Theorem~\ref{thm:s=2-r=2-alternate}).  Additionally, compared to Boros's, our criteria for multistationarity are simpler, and we also interpret them geometrically via reaction diagrams (as in the above theorems).

In addition to applying Boros's results, 
we also rely on existing criteria from chemical reaction network theory~\cite{FeinOsc,HornJackson72} (specifically, deficiency theory) that can preclude or guarantee
multistationarity for certain classes of networks.  Indeed, these criteria can be viewed as the first partial answers to Question~\ref{q:main}.  For a survey,
see~\cite{mss-review}. 

Besides using standard criteria from chemical reaction network theory, we also employ concepts (e.g., Newton polytopes) and basic results (e.g., Descartes' rule of signs) from real algebraic geometry.  This is unsurprising: Question~\ref{q:main} and related problems are algebraic in nature~\cite{invitation,GW}, and criteria for real root counting are in terms of signs~\cite{BPR}.  
Accordingly, criteria for deciding multistationarity based on sign conditions have a long history 
\cite{BP,FeinDefZeroOne,signs}.

In particular, we highlight the role of Newton polytopes and their geometric relationships to the reaction vectors of a network.   The Newton polytope (or ``reactant polytope'') we are interested in is the convex hull of the reactant (source) vectors.  For instance, for the network $\{A+2B \to 3B, ~ 3A+B \to 4A \}$, the polytope is the marked diagonal of the rectangle below:
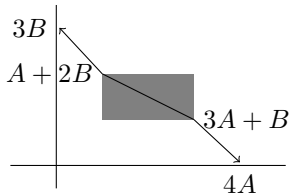
\begin{center}
	\begin{tikzpicture}[scale=.61]
	\draw (-1,0) -- (5,0);
	\draw (0,-.5) -- (0,3.5);
	\draw [->] (1,2) -- (0.07,3);
	\draw [->] (3,1) -- (4,0.07);
	\path [fill=gray] (1,1) rectangle (3,2);
	\draw (1,2) --(3,1); 
    \node [left] at (1,2) {$A+2B$};
    \node [left] at (0,3) {$3B$};
    \node [right] at (3,1) {$3A+B$};
    \node [below] at (4,0) {$4A$};
	\end{tikzpicture}
\end{center}
As an example, the above reaction diagram does not match any of the four forms depicted earlier, 
so we immediately conclude that the network is {\em not} multistationary.

For readers already familiar with real algebraic geometry, it is no surprise that these geometric objects appear in our analyses.  Additionally, for someone familiar with the ideas in Craciun's recent proof of (and others' prior partial results toward) the so-called global attractor conjecture~\cite{Anderson,GAC,CNP,GMS2,Pantea}, it is well known that analyzing the Netwon polytope of a network elucidates some aspects of the long-term dynamics and can be used to 
determine whether the network always admits {\em at least one} positive steady state.  What is new here is our use of the geometric objects to determine whether a network admits {\em more than one} positive steady state.

Our work is organized as follows.
Section~\ref{sec:background} introduces chemical reaction networks and the dynamical systems they define. 
Section~\ref{sec:preliminary-results} completes the simplest cases: when a network has only one species or only one reaction.
There we also show the following:
letting $r$ denote the number of reactions and $s$ the number of species,
 nondegenerately multistationary networks always satisfy the inequality $r+s \geq 4$.
Section~\ref{sec:main-theory} treats the next case: networks with two reactions and two species, and then Section~\ref{sec:two-rxn-net} extends these results to arbitrary networks with two reactions, each of which is possibly reversible.
Finally, a discussion appears in Section~\ref{sec:openQ}.

\section{Background} \label{sec:background}
In this section we recall how a chemical reaction network gives rise to a dynamical system.  

We begin with an example of a {\em chemical reaction}: $A+B \to 3A + C$.  
%
In this reaction, one unit of chemical {\em species} $A$ and one of $B$ react 
to form three units of $A$ and one of $C$.  
The  {\em reactant} $A+B$ and the {\em product} $3A+C$ are called {\em complexes}. 
The concentrations of the three species, denoted by $x_{A},$ $x_{B}$, and
$x_{C}$, change in time as the reaction occurs.  Under the assumption of {\em
mass-action kinetics}, species $A$ and $B$ react at a rate proportional to the
product of their concentrations, where the proportionality constant is the {\em reaction rate
constant} $\kappa$.  Noting that the reaction yields a net change of two units in
the amount of $A$, we obtain the differential equation $\frac{d}{dt}x_{A}=2\kappa x_{A}x_{B}$, and the other two arise similarly:
$\frac{d}{dt}x_{B} =-\kappa x_{A}x_{B}$ and 
$\frac{d}{dt} x_{C}=\kappa x_{A}x_{B}$.
 A {\em chemical reaction network}
consists of finitely many reactions (Definition~\ref{def:crn}).  
The mass-action differential equations that a network defines consist of a 
sum of the monomial contributions from the reactants of each 
chemical reaction in the network (see~(\ref{eq:ODE-mass-action})).


\subsection{Chemical reaction systems}
We now provide precise definitions.  

\begin{definition} \label{def:crn}
A {\em chemical reaction network} $G=(\SS,\CC,\RR)$
consists of three finite sets:
\begin{enumerate}
\item a set of chemical {\em species} $\SS = \{A_1,A_2,\dots, A_s\}$, 
\item a set  $\CC = \{y_1, y_2, \dots, y_p\}$ of {\em complexes} (finite nonnegative-integer combinations of the species), and 
\item a set of {\em reactions}, which are ordered pairs of the complexes, excluding the diagonal pairs: $\RR \subseteq (\CC \times \CC) \setminus \{ (y,y) \mid y \in \CC\}$.
\end{enumerate}
A {\em subnetwork} of a network $G = (\SS, \CC, \RR)$ is a network $G' = (\SS', \CC', \RR')$ with $\RR' \subseteq \RR$.
\end{definition}

\noindent
In this work, $p$, $s$, and $r$ denote the numbers of
complexes, species, and reactions, respectively.  

We may view a network as a directed graph: the nodes are complexes, and the edges correspond to reactions. A {\em linkage class} is a connected component of the directed graph: complexes $y$ and $y'$ are in the same linkage class if there is a sequence of complexes $(y_0:=y,~y_1, \ldots,~y_n:=y')$ such that $y_i \to y_{i+1}$ or $y_{i+1} \to y_{i}$ is a reaction for all $0 \le i \le n-1$. 
A network is \emph{weakly reversible} if every connected component 
is strongly connected, i.e., every reaction is part of a directed cycle of reactions.  A reaction $y_i \to y_j$ is {\em reversible} if its reverse reaction $y_j \to y_i$ is also in $\RR$; such a pair is depicted by $y_i \rightleftharpoons y_j$.

Writing the $i$-th complex as $y_{i1} A_1 + y_{i2} A_2 + \cdots + y_{is}A_s$ (where $y_{ij} \in \mathbb{Z}_{\geq 0}$ for $j=1,2,\dots,s$), 
we introduce the following monomial:
$$ x^{y_i} \,\,\, := \,\,\, x_1^{y_{i1}} x_2^{y_{i2}} \cdots  x_s^{y_{is}}~. $$
For example, the two complexes in the reaction $A+B \to 3A + C$ considered earlier give rise to 
the monomials $x_{A}x_{B}$ and $x^3_A x_C$, which determine the vectors 
$y_1=(1,1,0)$ and $y_2=(3,0,1)$.  
These vectors define the rows of a $p \times s$-matrix of nonnegative integers,
which we denote by $Y=(y_{ij})$.
Next, the unknowns $x_1,x_2,\ldots,x_s$ represent the
concentrations of the $s$ species in the network,
and we regard them as functions $x_i(t)$ of time $t$.

For a reaction $y_i \to y_j$ from the $i$-th complex to the $j$-th
complex, the {\em reaction vector}
 $y_j-y_i$ encodes the
net change in each species that results when the reaction takes
place.  The {\em stoichiometric matrix} 
$\Gamma$ is the $s \times r$ matrix whose $k$-th column 
is the reaction vector of the $k$-th reaction 
i.e., it is the vector $y_j - y_i$ if $k$ indexes the 
reaction $y_i \to y_j$.

We associate to each reaction 
a positive parameter $\kappa_{ij}$, the rate constant of the
reaction.  In this article, we will treat the rate constants $\kappa_{ij}$ as positive
unknowns in order to analyze the entire family of dynamical systems
that arise from a given network as the $\kappa_{ij}$'s vary.  


The choice of kinetics is encoded by a locally Lipschitz function $R:\Rnn^s \to \R^r$ that encodes the reaction rates of the $r$ reactions as functions of the $s$ species concentrations.
The {\em reaction kinetics system} 
defined by a reaction network $G$ and reaction rate function $R$ is given by the following system of ODEs:
\begin{align} \label{eq:ODE}
\frac{dx}{dt} ~ = ~ \Gamma \cdot R(x)~.
\end{align}
For {\em mass-action kinetics}, which is the setting of this paper, the coordinates of $R$ are
$ R_k(x)=  \kappa_{ij} x^{y_i}$, 
 if $k$ indexes the reaction $y_i \to y_j$.  
A {\em chemical reaction system} refers to the 
dynamical system (\ref{eq:ODE}) arising from a specific chemical reaction
network $(\SS, \CC, \RR)$ and a choice of rate parameters $(\kappa^*_{ij}) \in
\mathbb{R}^{r}_{>0}$ (recall that $r$ denotes the number of
reactions) where the reaction rate function $R$ is that of mass-action
kinetics.  Specifically, the mass-action ODEs are the following:
\begin{align} \label{eq:ODE-mass-action}
\frac{dx}{dt} \quad = \quad \sum_{ y_i \to y_j~ {\rm is~in~} \RR} \kappa_{ij} x^{y_i}(y_j - y_i) \quad =: \quad f_{\kappa}(x)~.
\end{align}

The {\em stoichiometric subspace} is the vector subspace of
$\mathbb{R}^s$ spanned by the reaction vectors
$y_j-y_i$, and we will denote this
space by $\St$: 
\begin{equation} \label{eq:stoic_subs}
  \St~:=~ {\rm span} \left( \{ y_j-y_i \mid  y_i \to y_j~ {\rm is~in~} \RR \} \right)~.
\end{equation}
Note that $\St = \im(\Gamma)$, where $\Gamma$ is the stoichiometric matrix.
For the network consisting of the single reaction $A+B \to 3A + C$, we have $y_2-y_1 =(2,-1,1)$,  which spans $\St$ for this network. 

Note that the  vector $\frac{d x}{dt}$ in  (\ref{eq:ODE}) lies in
$\St$ for all time $t$.   
In fact, a trajectory $x(t)$ beginning at a positive vector $x(0)=x^0 \in
\R^s_{>0}$ remains in the following {\em stoichiometric compatibility class}:
\begin{align}\label{eqn:invtPoly}
\invtPoly~:=~(x^0+\St) \cap \mathbb{R}^s_{\geq 0}~
\end{align}
for all positive time.  In other words, $\invtPoly$ is forward-invariant with
respect to the dynamics~(\ref{eq:ODE}).    

\subsection{Steady states} \label{subsec:steady-states}

A {\em steady state} of a reaction kinetics system is a nonnegative concentration vector $x^* \in \Rnn^s$ at which the ODEs~\eqref{eq:ODE}  vanish: $f_{\kappa} (x^*) = 0$.  
A steady state $x^*$ is {\em nondegenerate} if ${\rm Im}\left( df_{\kappa} (x^*)|_{S} \right) = \St$. (Here, $df_{\kappa}(x^*)$ is the Jacobian matrix of $f_{\kappa}$ at $x^*$.)  
A nondegenerate steady state is 
{\em exponentially stable} if each of the $\sigma:= \dim(\St)$ nonzero eigenvalues of $df_{\kappa}(x^*)$ has negative real part. 
This work focuses on counting {\em positive steady states} $x ^* \in \mathbb{R}^s_{> 0}$.

\begin{definition} \label{def:mss}
\begin{enumerate}
\item  A reaction kinetics system~\eqref{eq:ODE} is {\em multistationary} (or {\em admits multiple positive steady states}) if there exists a stoichiometric compatibility class~\eqref{eqn:invtPoly} 
with two or more positive steady states. 
Similarly, a reaction kinetics system is {\em nondegenerately multistationary} (respectively, {\em multistable}) if it admits two or more nondegenerate (respectively, exponentially stable) positive steady states in some stoichiometric compatibility class~\eqref{eqn:invtPoly}.  

\item Under mass-action kinetics~\eqref{eq:ODE-mass-action}, a network may admit multiple positive steady states for all, some, or no choices of
positive rate constants $\kappa_{ij}$; if such rate constants exist, then the network itself 
{\em has the capacity for multistationarity} or, for short, 
is {\em multistationary}.  Analogously, a network can be {\em nondegenerately multistationary} or {\em multistable}.

\item 
A network {\em admits $k$ positive steady states} if there exists a choice of positive rate constants so that the resulting mass-action system 
has exactly $k$ positive steady states in some stoichiometric compatibility class.
Similarly, a network may {\em admit $k$ nondegenerate positive steady states} or {\em admit $k$ exponentially stable positive steady states}.\footnote{We allow $k=\infty$, by which we mean infinitely many steady states of the corresponding type.}

\item 
The 
{\em maximum number of positive steady states} of a network $G$ is 
$$\capPSS(G) ~:=~ \max \{k \in\mathbb{Z}_{\geq 0} \cup \infty \mid G {\rm~admits~} k {\rm~positive~steady~states} \} ~.$$
The {\em maximum number of nondegenerate positive steady states} of a network $G$ is 
$$ \capNPSS (G) ~:=~ \max \{k \in\mathbb{Z}_{\geq 0} \cup \infty \mid G {\rm~admits~} k {\rm~nondegenerate~positive~steady~states} \} ~,$$
which we also call the the {\em capacity} of $G$. 
Analogously, $G$ has 
a {\em maximum number of exponentially stable positive steady states}, denoted by 
$\capStable(G)$. 
\end{enumerate}

\end{definition}

\noindent
Note that a network $G$ is multistationary (respectively, nondegenerately multistationary or multistable) if and only if $\capPSS(G)>1$ (repectively, $\capNPSS(G)>1$ or $\capStable(G)>1$ ).  
Stable steady states are nondegenerate, so the following inequalities hold for all networks $G$:
	\begin{align*} 
\capPSS(G)  \ge \capNPSS(G)  \geq \capStable(G)~.
	\end{align*}
It is natural to ask under what hypotheses do these inequalities become equality: 

\begin{conjecture}[Nondegeneracy conjecture] \label{conj:equality}
If $\capPSS(G) < \infty$, then  $\capPSS(G) =  \capNPSS(G)$.
\end{conjecture}

The results in this work are consistent with Conjecture~\ref{conj:equality}, but we can not prove it in general.  The following example demonstrates why the $\capPSS(G) < \infty$ assumption is necessary.

\begin{example}[Infinitely many degenerate steady states]  \label{ex:cap} 
Let $G$ denote the simple network $0 \leftarrow A \to 2A$. 
If the two rate constants are equal, then every positive value $x_A>0$ is a degenerate steady state  (so, $\capPSS(G)=\infty$), whereas when the two rate constants are not equal, then the resulting system admits no positive steady states (so, $\capNPSS(G) = 0$).  

For a more interesting example, consider the following network: 
\begin{align*}
	G'~=~\{	A+B \to 0~, \quad \quad 2A \to 3A+B\}~.
\end{align*}
The mass-action ODEs are:
\begin{align*}
\frac{dx_A}{dt} ~=~ \frac{dx_B}{dt}~=~ -k x_Ax_B + lx_A^2~,
\end{align*}
so the compatibility classes are $\invtPoly=\{(a,a+T) \in \mathbb{R}^2_{\geq 0} \}$, where $T \in \mathbb{R}$.  Thus, the positive steady states come from the positive roots of $0=-k(x_A+T)+lx_A=(-k+l)x_A-kT$.  So, if $k \neq l$, there is at most 1 positive steady state (arising from $x_A^*=kT/(l-k)$, if $l >k$).  On the other hand, if $k=l$, then no positive steady states exist if $T \neq 0$, whereas the entire  compatibility class consists of steady states if $T =0$.  Thus,  $G'$ admits infinitely many degenerate positive steady states (so, $\capPSS(G')=\infty$), but at most 1 nondegenerate one (so, $\capNPSS(G') = 1$).  We will recover this result later (part 3b of Theorem~\ref{thm:s=2-r=2-alternate}); see Example~\ref{ex:-1-slope}.
\end{example}

Now we can state a shorter version of Question~\ref{q:main}: {\em which reaction networks are multistationary? }
As mentioned in the introduction, this question is difficult in general: assessing whether a given network is multistationary means determining if the parametrized family of polynomial systems arising from the mass-action ODEs~\eqref{eq:ODE} ever admits two or more positive solutions.  

We also pose the following related questions, because in applications we are most interested in the steady states that we can observe:
\begin{question} \label{q:main-nondeg}
Which reaction networks are nondegenerately multistationary?
\end{question}
\begin{question} \label{q:main-multistable}
Which reaction networks are multistable?
\end{question}
\noindent
Other authors have posed such questions (e.g.,~\cite{BP,FeinOsc}).  Nevertheless, here we give the first complete answer for the smallest networks: those with one reaction or one species (Section~\ref{sec:preliminary-results}), two reactions and two species (Section~\ref{sec:main-theory}), and two (possibly reversible) reactions (Section~\ref{sec:two-rxn-net}).

Now we recall a well-known fact\footnote{For instance, see~\cite[Remark 2.1]{Fein95DefOne} or \cite[Lemma~D.2]{BP}.}: the networks that admit positive steady states 
are the consistent ones. A network is {\em consistent} if its reaction vectors are positively linearly dependent. 

\begin{lemma}  \label{lem:consistent}
Let $G$ be a network.  Then 
$\capPSS(G) \ge 1$ if and only if $G$ is consistent. 
\end{lemma}


Lemma~\ref{lem:consistent} implies that if a 2-reaction network admits positive steady states, then its reaction vectors are negative scalar multiples of each other (we will state this in Lemma~\ref{lem:2-rxn-dep}).  
For instance, $\{A+D \to B+D, ~ 2A+D \to C+D \}$ 
from~\cite[Example 10]{BP} 
is not multistationary.

One important subclass of consistent networks is formed by the weakly reversible networks.  In fact, Deng~\textit{et al.}\ showed that weakly reversible networks admit positive steady states (in every stoichiometric compatibility class, for all choices of rate constants)~\cite{deng}.

\subsection{The deficiency of a \CRN} \label{subsec:deficiency}
The  {\em deficiency} $\delta$ of a reaction network is an important invariant.   
To define it, recall that $p$ denotes the number
of complexes of a network, and let $l$ denote 
the number of linkage classes (connected components). For most networks in this work, each linkage class contains a unique {\em terminal strong linkage class}, i.e.~a maximal strongly connected subgraph in which there are no reactions from a complex in the subgraph to a complex outside the subgraph. 
In this case, Feinberg showed that the deficiency of the network can be computed in the following way: 
\begin{align} \label{def-formula}
 \delta~:=~p-l-\dim(\St)~,
\end{align}
where $\St$ denotes the stoichiometric subspace~\eqref{eq:stoic_subs}.  
The deficiency of a reaction network is nonnegative
because it can be interpreted as the dimension of a
certain linear subspace \cite{Feinberg72}.  

The foremost results in deficiency theory are the following, which are due to Feinberg~\cite{FeinDefZeroOne}: 
\begin{lemma}[Deficiency~zero theorem] \label{lem:thm:def-0}
Deficiency-zero networks are not multistationary:
\begin{enumerate}
\item Every zero-deficiency network that is weakly reversible admits a unique positive steady state in every stoichiometric compatibility class (for any choice of rate constants), and this steady state is locally asymptotically stable.
\item Every zero-deficiency network that is {\em not} weakly reversible admits {\em no} positive steady states (for any choice of rate constants).
\end{enumerate} 
\end{lemma} 

\begin{lemma}[Deficiency~one theorem] \label{lem:thm:def-1}
Consider a network $G$ with linkage classes $G_1$, $G_2$, \dots, $G_l$.  Let $\delta$ denote the deficiency of $G$, and let $\delta_i$ denote the deficiency of $G_i$.  Assume that:
\begin{enumerate}
\item each $G_i$ has only one terminal strong linkage class (this holds if $G_i$ is weakly reversible),
\item $\delta_i \leq 1$ for all $i=1,2,\dots,l$, and
\item $\delta_1+\delta_2+ \dots + \delta_l = \delta$. 
\end{enumerate}
Then $G$ is not multistationary: every resulting mass-action system admits either 0 positive steady states or exactly 1 in every stoichiometric compatibility class.
\end{lemma}

\subsection{Minimal multistationary networks} \label{sec:atoms}

We now recall the definition of ``embedding'' and how this notion is useful for assessing multistationarity.
Embedded networks were introduced by the authors in~\cite{simplifying}. 
The embedding relation generalizes the subnetwork relation --  a subnetwork $N$ is obtained from a reaction network $G$ by removing a subset of reactions (that is, setting some of the reaction rates to 0), while an embedded network is obtained by removing a subset of reactions and/or species.
For instance, removing the species $B$ from the reaction $A+B \to A+C$ results in the reaction $A \to A+C$.

\begin{definition} Given a set of reactions $\RR$ and a set of species $\SS$, the {\em restriction of $\RR$ to $\SS$}, denoted $\RR |_{\SS}$, is the subset of $\RR$ remaining after performing the following steps: (1) set to 0 the stoichiometric coefficients of all species 
not in $\SS$, and then (2) discard any {\em trivial reactions} (reactions of the form $\sum m_i A_i \to \sum p_j A_j$, i.e., when the source complex equals the product).
\end{definition}

\begin{definition} \label{def:emb}
The {\em embedded network} $N$ of a network $G = (\SS,\CC,\RR)$ 
 obtained by removing the set of reactions $\{ y \to y' \} \subseteq \RR$ and the set of species  $\{X_i\} \subseteq \SS$ is
\[
N = \left(\SS|_{\CC|_{\RR_N}}, \CC|_{\RR_N}, \ \RR_N := \left(\RR \setminus  \{ y \to y' \}\right)|_{\SS \setminus \{X_i\}}\right)~.
\]
\end{definition}

If $N$ is embedded in $G$, then under some hypotheses, $\capNPSS(N) \leq \capNPSS(G)$ 
(see the survey~\cite[\S 4]{mss-review}). 
One such result, due to Joshi and Shiu~\cite[Theorem~3.1]{joshi2012atoms},``lifts''  steady states from a subnetwork to a larger network if they share the same stoichiometric subspace:

\begin{lemma} \label{lem:lift}
Let $N$ be a subnetwork of a reaction network $G$ which has the same stoichiometric subspace as $G$.  Then $\capNPSS(N) \leq \capNPSS(G) $ and  $\capStable(N) \leq \capStable(G) $.
\end{lemma} 

This motivates the search for the minimal multistationary networks: 

\begin{definition} \label{def:atom}
\been
\item A network is an {\em embedding-minimal multistationary network} if, with respect to the embedded network relation, it is minimal among all networks that admit multiple nondegenerate positive steady states.
\item An {\em embedding-minimal multistable network} is minimal with respect to the embedded network relation among all networks that admit multiple stable positive steady states.
\enen
\end{definition}

The embedding-minimal multistationary networks previously discovered are (1) certain networks with 1 or 2 species and all inflow/outflow reactions~\cite[Theorem 6.1]{mss-review}, and (2) some ``sequestration'' networks~\cite[Theorem 6.11]{mss-review}\footnote{The authors conjectured in~\cite[\S 6.3]{mss-review} that these sequestration networks are embedding-minimal multistationary networks, and subsequently F\'elix, Shiu, and Woodstock resolved this for the smallest cases~\cite{FSW}.}.
One of our contributions here is to add to this list of known embedding-minimal multistationary networks:
\begin{theorem} \label{thm:list-new-atoms}
Among all networks $G$ that have exactly 1 species or 
that are subnetworks of a 2-reversible-reaction network (i.e., $\{y \to y' \} \subseteq G \subseteq \{y \lra y',~ \widetilde y \lra \widetilde y'\}$), 
the embedding-minimal multistationary 
ones
 are precisely the following:
	\begin{enumerate}
	\item the 1-species networks with arrow diagram\footnote{``Arrow diagrams'' are defined in Definition~\ref{def:arrow-diagram}.} $(\to, \leftarrow, \to)$ or $(\leftarrow, \to, \leftarrow)$, 
	\item the 2-species, 2-reaction networks whose reaction diagram has one of the following forms:
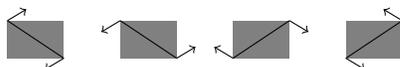
\begin{center}
	\begin{tikzpicture}[scale=.5]
	\path [fill=gray] (0,0) rectangle (1.5,1);
	\draw [->] (0,1) -- (.5,1.3);
	\draw [->] (1.5,0) -- (1,-.3);
	\draw (0,1) --(1.5,0); 
	\path [fill=gray] (3,0) rectangle (4.5,1);
	\draw [->] (3,1) -- (2.5,.7);
	\draw [->] (4.5,0) -- (5,.3);
	\draw (3,1) --(4.5,0); 
	\path [fill=gray] (6,0) rectangle (7.5,1);
	\draw [->] (6,0) -- (5.5,.3);
	\draw [->] (7.5,1) -- (8, .7);
	\draw (6,0) --(7.5,1); 
	\path [fill=gray] (9,0) rectangle (10.5,1);
	\draw [->] (9,0) -- (9.5, -.3);
	\draw [->] (10.5,1) -- (10,1.3);
	\draw (9,0) --(10.5,1);
	\end{tikzpicture}
\end{center}
and the slope of the marked diagonal is {\em not} $-1$, and
	\item the 3-species networks for which there are two pairs of species, denoted $\{i_1,j_1\}$ and $\{i_2,j_2\}$, such that (for $k=1,2$) the projection of the reaction diagram to the $(i_k,j_k)$-plane has one of the following forms:
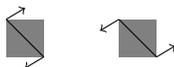
\begin{center}
	\begin{tikzpicture}[scale=.5]
	\path [fill=gray] (0,0) rectangle (1,1);
	\draw [->] (0,1) -- (.5,1.3);
	\draw [->] (1,0) -- (.5,-.3);
	\draw (0,1) --(1,0); 
	\path [fill=gray] (3,0) rectangle (4,1);
	\draw [->] (3,1) -- (2.5,.7);
	\draw [->] (4,0) -- (4.5,.3);
	\draw (3,1) --(4,0); 
	\end{tikzpicture}
\end{center}	
and the slope of the marked diagonal equals $-1$.	
	\end{enumerate}
\end{theorem}
\noindent
Theorem~\ref{thm:list-new-atoms} summarizes Corollaries~\ref{cor:s=1-r=1-atoms}, \ref{cor:s=2-r=2-atoms}, \ref{cor:r=2-atoms}, \ref{cor:i=1-r=1-atoms}, and \ref{cor:rev=2-atoms}.

\section{Networks with only one reaction or only one species} \label{sec:preliminary-results}

Theorem~\ref{thm:s=1-r=1} in this section answers Questions~\ref{q:main} and
\ref{q:main-nondeg}--\ref{q:main-multistable}
 for two easy cases: when there is only one reaction or one species.  
As a preview to the theorem, we now give some examples.  
\begin{example}  \label{ex:1-reaction}
For the network consisting of the single reaction 
$A+B \to 2A$, the mass-action ODEs~\eqref{eq:ODE-mass-action} are 
\[ \frac{dx_A}{dt} ~=~ - \kappa x_A x_B ~=~ - \frac{dx_B}{dt} ~. 
\]
So, $x_A=0$ or $x_B=0$ at steady state, and therefore there are no positive steady states. 
Clearly, such an argument generalizes for any 1-reaction network: see Theorem~\ref{thm:s=1-r=1} (part 1) below.
\end{example}

To motivate the second case -- networks with only one species -- we recall the following result, which is part of Joshi's classification of fully open networks (i.e., with all inflows/outflows) and one non-flow reaction~\cite{joshi2013complete}:
\begin{proposition} \label{prop:1-rxn-atoms}
The (general) fully open network with one species and one non-flow reaction:
 \[
	0 \leftrightarrow A \quad \quad \quad m_1 A \to n_1 A
\]
is multistationary if and only if $1 < m_1 < n_1$.
\end{proposition}
In light of Proposition~\ref{prop:1-rxn-atoms}, what distinguishes the multistationary network $0 \leftrightarrow A,~ 2A \to 3A$ from the non-multistationary network  $0 \leftrightarrow A,~ 2A \leftarrow 3A$?  Also, can this distinction be generalized beyond fully open networks?  The answers are as follows. 
First, the reactions in the first network have the form $(\to, \leftarrow, \to)$ (as seen in $0 \to A, ~ 0 \leftarrow A, ~ 2A \to 3A$, where we list the reactions so that the reactants are in increasing order) vs.\ $(\to, \leftarrow, \leftarrow)$ in the second network.  
And, indeed, these ideas generalize (see Definition~\ref{def:2-sign-change} and part 2 of Theorem~\ref{thm:s=1-r=1}). 

A final example is the 1-species network $0 \leftarrow A \to 2A$ from Example~\ref{ex:cap} that admits infinitely many (degenerate) positive steady states.  We will see that this property (admitting infinitely many degenerate positive steady states) has a simple characterization in the 1-species case: 
each reaction has an ``opposing'' reaction so the pair has the form ``$\lradot$"; see part 2 of Theorem~\ref{thm:s=1-r=1}.


\subsection{Main result} \label{subsec:main-result}
To state the main result of this section (Theorem~\ref{thm:s=1-r=1}), we need the following definitions.

\begin{definition} \label{def:arrow-diagram}
Let $G$ be a reaction network that contains only one species $A$. Thus, each reaction of $G$ has the form $aA \to bA$, where $a,b \ge 0$ and $a \ne b$. Let $m$ be the number of (distinct) reactant complexes, and let $a_1< a_2 < \ldots < a_m$ be their stoichiometric coefficients. The {\em arrow diagram of $G$}, denoted $\rho = (\rho_1, \ldots , \rho_m)$, is the element of $\{\to , \la, \lradot \}^m$ defined by:
\begin{equation*}
 \rho_i~=~ 
 \left\lbrace\begin{array}{ll}
   \to & \text{if for all reactions $a_iA \to bA$ in $G$, it is the case that $b > a_i$} \\
   \la & \text{if for all reactions $a_iA \to bA$ in $G$, it is the case that $b < a_i$} \\
   \lradot & \text{otherwise.}
 \end{array}\right.
\end{equation*}
\end{definition}

\begin{definition} \label{def:2-sign-change}
For positive integers $T \geq 2$, a {\em $T$-alternating network} is a 1-species network with 
exactly $T+1$ reactions and with arrow diagram $\rho \in \{\to , \la\}^{T+1}$ such that $\rho_i = \to$ if and only if $\rho_{i+1} = \la$ for all $i \in \{1, \ldots, T\}$. 
\end{definition}


\begin{example} \label{ex:alternating}
Consider the following network:
 	\begin{align*}
	G~=~\{ 0 \leftarrow A \to 2A \lra 3A \to A\}~.
	\end{align*}
\begin{enumerate}
	\item $\{ 0 \leftarrow A,~ 2A \to 3A\}$ is a 1-alternating subnetwork of $G$ 
	with arrow diagram $(\leftarrow, \to)$.
	\item Each of the following is a 1-alternating subnetwork of $G$ 
	with arrow diagram $(\to, \la)$:
	$\{ A \to 2A,~ 2A \leftarrow 3A\}$, 
	$\{ A \to 2A,~ A \leftarrow 3A\}$, and 
	$\{  2A \to 3A,~ 2A \leftarrow 3A\}$.
	\item $\{ 0 \leftarrow A,~ A \to 2A\}$ is {\em not} a 1-alternating subnetwork of $G$: 
	its arrow diagram is $(\lradot)$.
	\item $\{ 0 \leftarrow A,~ 2A \to 3A, ~ 2A \leftarrow 3A \}$ is a 2-alternating subnetwork of $G$ with arrow diagram $(\la, \to, \la)$.
\end{enumerate}
\end{example}


\begin{theorem}[Classification of multistationary networks with only one reaction or species] \label{thm:s=1-r=1}
~
\begin{enumerate}
	\item Every network that consists of a single (irreversible) reaction ($y \to y'$) or a single pair of reversible reactions ($y \leftrightarrows y'$) is {\em not} multistationary. 
	\item Let $G$ be a network with only one species.  Then $G$ is multistationary (i.e., $\capPSS(G)\geq 2$) if and only if  $G$ has a 2-alternating subnetwork or 
	 the arrow diagram of $G$ has the form $(\lradot, \dots, \lradot)$. 
Moreover:
	\begin{enumerate}[(a)]
		\item $G$ admits infinitely many (degenerate) positive steady states (i.e., $\capPSS(G) = \infty$) if and only if 
		 the arrow diagram of $G$ has the form $(\lradot, \dots, \lradot)$. 
		\item For $T \geq 2$, $G$ admits at least $T$ nondegenerate positive steady states (i.e., $\capNPSS(G) \geq T$) if and only if $G$ has a $T$-alternating subnetwork.  
		\item For $T \geq 2$, $G$ admits at least $\lceil T/2 \rceil$ (respectively, $\lfloor T/2 \rfloor$) exponentially stable positive steady states (i.e., $\capStable(G) \geq \lceil T/2 \rceil$ or $ \geq \lfloor T/2 \rfloor$, respectively)  if $G$ has a $T$-alternating subnetwork 
		 with arrow diagram $(\to, \la, \to, \ldots)$ (respectively, $(\la, \to, \la, \ldots)$).
	\end{enumerate}
Thus, $G$ is nondegenerately multistationary if and only if $G$ has a 2-alternating subnetwork.
\end{enumerate}
\end{theorem}

We will prove Theorem~\ref{thm:s=1-r=1} in Section~\ref{sec:proof-s=1-r=1}.

\begin{remark} \label{rmk:boros}
In~\cite{revisiting}, Boros analyzed 1-species networks that consist of two reversible reactions $y \lra y'$ and $\widetilde y \lra \widetilde y'$ with no complexes in common, and determined that such a network is multistationary if and only if the network 
has arrow diagram $(\to, \la, \to, \la)$
(for instance, if $y < y' < \widetilde y < \widetilde y'$).
Here, Theorem~\ref{thm:s=1-r=1} easily recovers Boros's criterion.  See also Theorem~\ref{thm:rev=2}.
\end{remark}


\begin{corollary} \label{cor:r+s}
Every nondegenerately multistationary reaction network satisfies $r+s \geq 4$, where $r$ is the number of reactions and $s$ is the number of species.
\end{corollary}
\begin{proof}
We prove the contrapositive.  
The options for $(r,s)$ are $(1,1)$, $(1,2)$, and $(2,1)$.  If $r=1$, then the result follows from Theorem~\ref{thm:s=1-r=1} (part 1).  In the remaining case, the network has only one species and two reactions, so it does not have a 2-alternating subnetwork, and hence the result follows from  Theorem~\ref{thm:s=1-r=1} (part 2).
\end{proof}

The only obstacle to a partial converse of Corollary~\ref{cor:r+s} is the fact that 1-reaction networks are {\em not} multistationary.  That is, we have the following result, which we will prove in Section~\ref{sec:main-results-s=r=2}:
\begin{proposition} \label{prop:s+r}
If $r+s \geq 4$, with $r \geq 2$, then there exists a nondegenerately multistationary reaction network with exactly $r$ reactions and $s$ species. 
\end{proposition}

Returning to 1-species/1-reaction, Theorem~\ref{thm:s=1-r=1} directly implies 
the next two results: 
\begin{corollary} \label{cor:s=1}
For a reaction network $G$ that has only one species, the capacity of $G$ (i.e., the maximum number of nondegenerate positive steady states) is 
$$ \capNPSS (G)~=~ \max \{T \in \mathbb{Z}_{\geq 0} \mid G \mathrm{~has~a~}T\mathrm{-alternating~subnetwork} \}~,$$
and the maximum number of positive steady states is:
$$ \capPSS (G)~=~
  \begin{cases} 
      \infty  & \text{ \specialcell{the arrow diagram of $G$ has the form $(\lradot, \dots, \lradot)$.}}\\ 
      \capNPSS(G) & \text{else.} 
  \end{cases}
$$
\end{corollary}

We obtain new infinite families of 1-species minimal multistationary/multistable networks\footnote{Part 2 of Corollary~\ref{cor:s=1-r=1-atoms} includes a family of networks aready known to be embedding-minimal multistationary networks~\cite[Theorem 6.1 (part 2)]{mss-review}.}:
\begin{corollary}[Classification of embedding-minimal multistationary/multistable networks with one reaction or one species] \label{cor:s=1-r=1-atoms}
~
\been
\item There are no 1-reaction embedding-minimal multistationary or multistable networks. 
\item The 1-species embedding-minimal multistationary networks 
are precisely the 2-alternating networks (i.e., the  
3-reaction networks with arrow diagram $(\la, \to, \la)$ or $(\to, \la, \to)$). 
\item The 1-species embedding-minimal multistable networks are precisely the 3-alternating networks with arrow diagram $(\to, \la, \to, \la)$. 
\enen
\end{corollary}

Now we consider at-most-bimolecular networks.  
A complex is {\em at-most-bimolecular} if it is of the form $0, A, 2A$, or $A+B$.
In other words, the sum of the stoichiometric coefficients over all the species in the complex is at most 2. A reaction network is itself {\em at-most-bimolecular} if all its complexes are at-most-bimolecular. 

\begin{corollary} \label{cor:at-most-bi-s=1-r=1}
Suppose that $G$ is an at-most-bimolecular reaction network. If $G$ has only one reaction or only one species, then $G$ is {\em not} nondegenerately multistationary (i.e., $\capNPSS(G) \le 1$), and thus is {\em not} multistable. 
\end{corollary}
\begin{proof}
By Theorem \ref{thm:s=1-r=1} (part 1), if $G$ has only one reaction, it is not multistationary. If $G$ has only one species, then it is nondegenerately multistationary if and only if $G$ has a 2-alternating subnetwork by Theorem \ref{thm:s=1-r=1} (part 2). Assuming that $G$ is nondegenerately multistationary, $G$ must have a 2-alternating subnetwork $\{m_1 A \to n_1 A, ~m_2 A \to n_2 A, ~m_3 A \to n_3 A \}$ with $0 \le m_1 < m_2 < m_3 \le 2$. The only possible choice is 
$(m_1,m_2,m_3)=(0,1,2)$. 
It follows that $n_1 > m_1$,
so the arrow diagram of $G$ must be $(\to, \la, \to)$.
Hence, $n_3 > m_3 =2$, a contradiction. 
So, either $G$ is {\em not} at-most-bimolecular, or $G$ is {\em not} nondegenerately multistationary.  
\end{proof}

\subsection{Proof of Theorem~\ref{thm:s=1-r=1}} \label{sec:proof-s=1-r=1}
The proof of Theorem~\ref{thm:s=1-r=1} requires two preliminary results.

\begin{definition} \label{def:sign}
The {\em sign} of a real number $a \in \mathbb{R}$ is
\[
 {\rm sign}(a) :=
  \begin{cases} 
      + & \text{ if $a>0$} \\
      0 & \text{ if $a=0$} \\
      - & \text{ if $a<0$}~. \\
  \end{cases}
\]
We define the {\em sign} of a vector $x \in \mathbb{R}^n$ component-wise: ${\rm sign}(x):=\left(  {\rm sign}(x_1),   {\rm sign}(x_2), \dots   {\rm sign}(x_n) \right) \in  \{+,0,-\}^n$.  
The {\em number of sign changes} in such a vector of signs $v \in \{+,0,-\}^n$ is obtained by first removing all 0's from $v$ and then counting the number of times in the resulting vector a coordinate switches from $+$ to $-$ or from $-$ to $+$.  For instance, the vector $(+,-,-,0,-,0,+)$ has two sign changes.
\end{definition}

\begin{lemma} \label{lem:T-alternating}
Let $G$ be a 1-species network, and let $T$ be a positive integer with $T \geq 2$.
\begin{enumerate}
	\item If there exists a choice of positive rate constants so that for the resulting mass-action ODE 
\begin{align} \label{eq:1-vble-ode}
	\dot x_A ~=~ c_0 + c_1 x_A + \cdots + c_m (x_A)^m~,
\end{align}
the vector of coefficients $\left( {\rm sign}(c_0), {\rm sign}(c_1), \dots, {\rm sign}(c_m) \right)$ has at least $T$ sign changes,
then 
 $G$ contains a T-alternating subnetwork.
	\item If there exists a choice of positive rate constants so that the resulting mass-action ODE~\eqref{eq:1-vble-ode} is the zero polynomial ($c_i=0$ for all $i$), then 
the arrow diagram of $G$ has the form $(\lradot, \dots, \lradot)$.
\end{enumerate}
\end{lemma}

\begin{proof}
Letting $r$ denote the number of reactions of $G$, we write the reactions as $p_j A \to q_j A$, for $j=1,2, \dots, r$.   Therefore, the coefficient $c_{m}$ in the mass-action ODE polynomial~\eqref{eq:1-vble-ode} is:
\begin{align} \label{eq:coef}
	c_{m} ~ =~ \sum_{ \{j \in \RR  \mid p_j = m \} } \kappa_j (q_j - m)~.
\end{align}

Let $\kappa_1, \kappa_2, \dots, \kappa_r>0$ be rate constants that satisfy the hypothesis of part 1 of the lemma.  Thus, there exist $T+1$ indices $0 \leq m_1 < m_2 <  \cdots < m_{T+1}$ such that the vector of coefficients $\left( {\rm sign}(c_{m_1}), {\rm sign}(c_{m_2}), \dots, {\rm sign}(c_{m_{T+1}}) \right)$ of the 
resulting mass-action ODE~\eqref{eq:1-vble-ode} is alternating: either $(+,-,+, \cdots)$ or $(-,+,-,\cdots)$.  
Examining the coefficient~$c_{m_i}$ as in~\eqref{eq:coef}, it follows that
there exists a reaction $m_i A \to q_j A$ with ${\rm sign}(c_{m_i}) = {\rm sign}(q_j - m_i)$.  
Thus, those $T+1$ reactions (from $i=1,2,\dots,T+1$) form a $T$-alternating subnetwork of $G$. 

Now let $\kappa_1, \kappa_2, \dots, \kappa_r>0$ be rate constants that satisfy the hypothesis of part 2 of the lemma; thus, $c_i=0$ for all $i$.  Consider any reactant complex $p_k A$.  The corresponding sum~\eqref{eq:coef} for $c_{p_k}$ 
must have at least one positive and at least one negative summand, 
so the corresponding entry of the arrow diagram is $\lradot$.  So, every entry in the arrow diagram of $G$ is $\lradot$.
\end{proof}

\noindent 
Both converses of Lemma~\ref{lem:T-alternating} hold; these are essentially contained in the proof of Theorem~\ref{thm:s=1-r=1}.

Lemma~\ref{lem:descartes-achieved} below, due to Grabiner~\cite{grabiner}, states that for polynomials whose coefficients alternate in sign, the bound from Descartes' rule can be achieved by some choice of coefficients. 

\begin{proposition}[Descartes' rule of signs -- bound only] \label{pro:Descartes_original}
Given a nonzero univariate
real polynomial $f(x) = c_0 + c_1 x + \cdots + c_r x^r$,
the number of positive real roots of $f$, counted with multiplicity, is bounded above by the number of sign variations
in the ordered sequence of the coefficients ${\rm sign}(c_0),\dots, {\rm sign}(c_r)$,
i.e., discard the $0$'s in this sequence and then count the number of times two consecutive signs differ. 
\end{proposition}

\begin{lemma}[Bound in Descartes' rule of signs achieved~\cite{grabiner}] \label{lem:descartes-achieved}
For $0 < l_1 < l_2 < \ldots < l_q$, consider the polynomial 
\begin{align*}
P(x) ~:=~ 1- k_1x^{l_1} + k_2x^{l_2} - \cdots + (-1)^q k_q x^{l_q}~.
\end{align*}
 Then there exist $k_1 >0,$ $k_2>0, \ldots, k_q >0$ such that $P(x)$ has $q$ positive roots $0<x^*_1<x^*_2<\cdots<x^*_q$, and 
$P'(x_i^*)<0$ for $i$ odd and $P'(x_i^*)>0$ for $i$ even. 
\end{lemma}
\begin{proof}
This follows immediately from~\cite[Theorem~1]{grabiner} and its proof.
\end{proof}

We now prove our classification theorem for networks with one species or one reaction.
\begin{proof}[Proof of Theorem~\ref{thm:s=1-r=1}]
Part 1 of the theorem follows from Lemma~\ref{lem:thm:def-0} (the deficiency zero theorem): such a network has deficiency zero.

To prove $\Rightarrow$ of part 2, assume that $G$ is multistationary.   Thus, there exist rate constants $\kappa_1, \kappa_2, \dots, \kappa_r>0$ so that the resulting mass-action ODE $\frac{dx_A}{dt}=:f(x_A)$, as in equation~\eqref{eq:1-vble-ode}, is a polynomial with multiple positive roots.  If $f$ is the zero polynomial, then by part~2 of Lemma~\ref{lem:T-alternating}, 
the arrow diagram of $G$ has the form $(\lradot, \dots, \lradot)$.
If $f$ is not the zero polynomial, then $\Rightarrow$ of part 2 follows from Lemma~\ref{lem:T-alternating} (part~1) and Descartes' rule of signs (Proposition~\ref{pro:Descartes_original}).  
Similarly, those two results prove $\Rightarrow$ of part 2(b).

For 2(a), the only polynomial with infinitely many positive roots is the zero polynomial, so if $G$ admits infinitely many positive steady states, they are degenerate and the corresponding mass-action ODE is the zero polynomial.  Thus, $\Rightarrow$ of 2(a) follows from part~2 of Lemma~\ref{lem:T-alternating}.  For $\Leftarrow$ of 2(a), if
the arrow diagram of $G$ has the form $(\lradot, \dots, \lradot)$,
 then each coefficient in the mass-action ODE is a sum of negative and positive summands (as in equation~\eqref{eq:coef}).  Therefore, the reaction rate constants $\kappa_j>0$ can be chosen so that each coefficient is zero, and therefore $G$ admits infinitely many degenerate positive steady states.

For 2(c) and $\Leftarrow$ of 2(b), we first show that every $T$-alternating network admits $T$ nondegenerate positive steady states, of which $\lceil {T}/{2} \rceil$ are exponentially stable 
when the arrow diagram is $(\to,\la,\to,\ldots)$,
or $\lfloor {T}/{2} \rfloor$ 
 when the arrow diagram is $(\la,\to,\la,\ldots)$.
For such a network, the mass-action ODE~\eqref{eq:ODE-mass-action} is:
 \begin{align*}
 \frac{dx_A}{dt} ~ = ~ \kappa_1 (n_1-m_1) (x_A)^{m_1} +  \kappa_2 (n_2-m_2) (x_A)^{m_2} + \cdots +  \kappa_{T+1} (n_{T+1}-m_{T+1}) (x_A)^{m_{T+1}}~,
 \end{align*}
where $m_1 < m_2 < \cdots < m_{T+1}$ and the vector of signs of the coefficients is either $(+,-,+, \cdots)$ (the  $(\to, \la, \to, \ldots)$ case) or $(-,+,-, \cdots)$ (in the other case).  For the $(+,-,+, \cdots)$  case, via the coordinate change $\tilde \kappa_i := \kappa_i \left| n_i - m_i \right|$, we need only find $\tilde \kappa_i >0$ such that 
\begin{align} \label{eq:reduced-eq}
	f(x)~:=~
	 \tilde \kappa_1 x^{m_1} - 
	 \tilde \kappa_2 x^{m_2}  + \cdots + (-1)^T 
	 \tilde \kappa_{T+1} x^{m_{T+1}} ~=~0
\end{align}
has $T$ positive roots $x \in \mathbb{R}_{>0}$, of which $\lceil T/2 \rceil$ satisfy $f'(x)<0$.  Indeed, this follows from Lemma~\ref{lem:descartes-achieved}.
Finally, the proof for the $(-,+,-, \cdots )$ case proceeds similarly ($T$ positive roots, and at $\lfloor T/2 \rfloor$ of them,  the derivative is negative).  Now, part 2(c) and $\Leftarrow$ of part 2(b) follow immediately from Lemma~\ref{lem:lift}.  Finally, $\Rightarrow$ of part 2 now follows from parts 2(b--c).
\end{proof}

\section{Networks with two reactions and two species} \label{sec:main-theory}
The previous section answered the questions of multistationarity and multistability 
when there is only one reaction or one species; here, Theorem~\ref{thm:s=2-r=2} answers these questions when there are two reactions and two species. 
Furthermore, Theorem~\ref{thm:s=2-r=2-alternate} characterizes the number and stability of these steady states.
Moreover, this criterion for multistationarity is geometric, via `box diagrams', and will in turn
generalize to the case of more species (Theorem~\ref{thm:speculation}).


Before stating results, we recall that only consistent networks admit positive steady states:

\begin{lemma} \label{lem:2-rxn-dep}
Let $G$ be a network with exactly two reactions, denoted by $y \to y'$ and $\widetilde y \to \widetilde y'$.   
If $G$ admits at least one positive steady state, then the reaction vectors are (strictly) negative scalar multiples of each other: $y'-y =- \lambda (\widetilde y' - \widetilde y)$ for some $\lambda>0$.
\end{lemma}

\begin{proof}
$G$ admits a positive steady state, so $G$ is consistent (Lemma~\ref{lem:consistent}). 
Thus, $\lambda_1 (y'-y) + \lambda_2 (\widetilde y' - \widetilde y)=(0,0)$ for some $\lambda_1,\lambda_2>0$, and, hence, $\lambda= \lambda_2/\lambda_1$ verifies the result.
\end{proof}

\subsection{Main results} \label{sec:main-results-s=r=2}

To state Theorem~\ref{thm:s=2-r=2}, we must define reactant polytopes (Newton polytopes) and box diagrams.

\begin{definition}[\cite{GMS2}] \label{def:NP}
The {\em reactant polytope} of a network $G$ is the convex hull of (i.e., the smallest convex set containing) the reactants of $G$ (in $\mathbb{R}^s$, where $s$ is the number of species).
\end{definition}

\begin{definition} \label{def:box}
Let $G$ be a network with exactly two species and two reactions, $y \to y'$ and $\tilde y \to \tilde y'$, such that the reactant vectors differ in both coordinates (i.e., writing $y=(y_A, y_B)$ and $\tilde y=(\tilde y_A, \tilde y_B)$, then both $y_A \neq \tilde y_A$ and  $y_B \neq \tilde y_B$).  The {\em \bd} of the network $G$ is the rectangle in $\mathbb{R}^2$ for which 
	\begin{enumerate}
	\item the edges are parallel to the axes of $\mathbb{R}^2$, and
	\item the reactants $y$ and $y'$ are two opposite corners of the rectangle.
	\end{enumerate}	
\end{definition}

\begin{remark}
The box diagram is the smallest rectangle containing the reactant polytope. 
\end{remark}

It will be useful to depict a box diagram together with the reaction vectors and the reactant polytope (which in this case is the diagonal of the box that connects the two reactants).  For example, consider the network $\{ B \to A , ~ A+2B \to 3B \}$, which is a subnetwork of a bistable network for modeling apoptosis due to Ho and Harrington~\cite{HoHarrington}.  Its box diagram is:
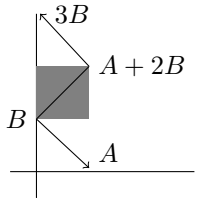
\begin{center}
	\begin{tikzpicture}[scale=.7]
	\draw (-0.5,0) -- (3,0);
	\draw (0,-.5) -- (0,3);
	\draw [->] (0,1) -- (1,0.07);
	\draw [->] (1,2) -- (0.07,3);
	\path [fill=gray] (0,1) rectangle (1,2);
	\draw (0,1) --(1,2); 
    \node [left] at (0,1) {$B$};
    \node [above right] at (1,0) {$A$};
    \node [right] at (1,2) {$A+2B$};
    \node [right] at (0,3) {~$3B$};
	\end{tikzpicture}
\end{center}

The aim of Theorem~\ref{thm:s=2-r=2} below is to give an easy-to-check criterion for multistationarity via box diagrams.  For instance, the diagram depicted above has the form of one of the four in~\eqref{eq:boxes}, so we quickly conclude that the network is multistationary (but not, it turns out, bistable).  Conversely, for a network with exactly two reactions and two species, if its box diagram does {\em not} have one of the forms depicted in~\eqref{eq:boxes}, then it is {\em not} nondegenerately multistationary.
For instance, the autocatalytic network $\{A \to B \to 2A \}$ 
from~\cite[Example 13]{BP} is {\em not} multistationary.

\begin{theorem}[Classification of nondegenerately multistationary networks with two reactions and two species] \label{thm:s=2-r=2}
Let $G$ be a network with exactly two species and two reactions, denoted by $y \to y'$ and $\widetilde y \to \widetilde y'$.  Then $G$ is nondegenerately multistationary if and only if the following hold:
\been
\item the reaction vectors are negative scalar multiples of each other, i.e., $y'-y =- \lambda (\widetilde y' - \widetilde y)$ for some $\lambda>0$, 
\item 
 the two reactants differ in both coordinates (so the box diagram of $G$ is defined), 
\item the slope of the reactant polytope is {\em not} -1,  and
\item  the \bd~of $G$ has one of the following zigzag forms:
\begin{center}
\begin{equation} \label{eq:boxes}
\begin{tikzpicture}[baseline=(current  bounding  box.center),scale=.6]
	\path [fill=gray] (0,0) rectangle (1.5,1);
	\draw [->] (0,1) -- (.5,1.3);
	\draw [->] (1.5,0) -- (1,-.3);
	\draw (0,1) --(1.5,0); 
	\path [fill=gray] (3,0) rectangle (4.5,1);
	\draw [->] (3,1) -- (2.5,.7);
	\draw [->] (4.5,0) -- (5,.3);
	\draw (3,1) --(4.5,0); 
	\path [fill=gray] (6,0) rectangle (7.5,1);
	\draw [->] (6,0) -- (5.5,.3);
	\draw [->] (7.5,1) -- (8, .7);
	\draw (6,0) --(7.5,1); 
	\path [fill=gray] (9,0) rectangle (10.5,1);
	\draw [->] (9,0) -- (9.5, -.3);
	\draw [->] (10.5,1) -- (10,1.3);
	\draw (9,0) --(10.5,1);
	\end{tikzpicture}
\end{equation}

\end{center}
\end{enumerate}
\end{theorem}
\noindent 
Theorem~\ref{thm:s=2-r=2} will largely follow from Theorem~\ref{thm:s=2-r=2-alternate} below, which goes beyond multistationarity to the precise number and nature of the steady states.  We postpone the proof until Section~\ref{sec:proofs-s=r=2}.

To state Theorem~\ref{thm:s=2-r=2-alternate}, we introduce the following notation:

\begin{definition} \label{def:B}
For a species $i$ and two reactions $y \to y'$ and $\widetilde y \to \widetilde y'$, we define
\begin{align*}
\beta_i(y \to y',~\widetilde y \to \widetilde y') ~:=~ (y'_i - y_i)(\widetilde y_i -  y_i)~\in~ \mathbb{Z}.
\end{align*}
Letting $s$ denote the number of species, we write 
\begin{align*}
\beta (y \to y',~\widetilde y \to \widetilde y') ~&:=~ \left( \beta_1(y \to y',~\widetilde y \to \widetilde y'), 
\dots, \beta_s(y \to y',~\widetilde y \to \widetilde y') \right) \\
& = ~(y' - y) \ast (\widetilde y - y) ~\in \mathbb{Z}^s ~,
\end{align*}
where $\ast$ denotes the Hadamard product, 
defined by $(u_1, \ldots, u_s) \ast (v_1 ,\ldots, v_s) := (u_1v_1,\ldots, u_sv_s)$.
\end{definition}

Lemma~\ref{lem:sign} below states that for a 2-reaction network, if the two reactions are positively dependent, then the 
sign of $\beta_i$ 
does not depend on the order of the two reactions, and therefore may be viewed as a function of the network itself.  We will 
see in Theorems~\ref{thm:s=2-r=2-alternate} 
and~\ref{thm:r=2}
that the signs of the $\beta_i$'s classify multistationary networks with two reactions and two species.

\begin{lemma} \label{lem:sign}
If  $y'-y =- \lambda (\widetilde y' - \widetilde y)$ for some $\lambda>0$, then 
	\begin{align*}
{\rm sign} (\beta_i(y \to y',~\widetilde y \to \widetilde y')) ~=~ {\rm sign} (\beta_i(\widetilde y \to \widetilde y',~y \to y'))~.
	\end{align*}
\end{lemma}
\begin{proof}
If $y'-y =- \lambda (\widetilde y' - \widetilde y)$, where $\lambda>0$, then 
${\rm sign}(\beta_i(\widetilde y \to \widetilde y',~y \to y')) = {\rm sign}((\widetilde y_i' - \widetilde y_i)  (y_i - \widetilde y_i)) = {\rm sign}\left( ( -1/ \lambda) ( y_i' -  y_i)  (y_i - \widetilde y_i) \right)  = {\rm sign}( ( y_i' -  y_i)  ( \widetilde y_i - y_i)) = {\rm sign}(\beta_i(y \to y',~\widetilde y \to \widetilde y'))$. 
\end{proof}

\begin{theorem}[Classification of multistationary networks with two reactions and two species] \label{thm:s=2-r=2-alternate}
Let $G$ be a network with exactly two species and two reactions, denoted by $y \to y'$ and $\widetilde y \to \widetilde y'$.   
For the two species $i=1,2$, let $\beta_i:=\beta_i(y \to y',~\widetilde y \to \widetilde y')$.  
If the reaction vectors are negative scalar multiples of each other: $y'-y =- \lambda (\widetilde y' - \widetilde y)$ for some $\lambda>0$, then we have the following cases:

\been

\item Suppose that $(\beta_1,\beta_2) = (0,0)$. Then $G$ admits the following possibilities and nothing else: (i) infinitely many degenerate positive steady states, (ii) no positive steady states. 

\item Suppose that $\beta_1=0$ or $\beta_2=0$, but not both.  We have the following subcases:
	\been
	\item If $y_1'=y_1$ and $\beta_2\neq 0$ (or, symmetrically, if $y_2'=y_2$ and $\beta_1\neq 0$), then $G$ admits exactly one nondegenerate positive steady state (in every stoichiometric compatibility class, for all choices of rate constants). 
	\item  In the remaining case --- when $y_1' \neq y_1$, $\widetilde y_1 = y_1$, and $\beta_2\neq 0$ (or, symmetrically, if $y_2' \neq y_2$, $\widetilde y_2 = y_2$, and $\beta_1\neq 0$) --- then $G$ admits the following possibilities and nothing else: (i) no  positive steady states, (ii) exactly one nondegenerate positive steady state. 
	\enen

\item Suppose that $\beta_1\neq 0$ and $\beta_2 \neq 0$.  We have the following subcases:
	\been
	\item If $\beta_1 \beta_2 > 0$, then $G$ admits exactly one positive steady state. 
	\item If $\beta_1  \beta_2 < 0$ and $\widetilde y - y$ 
is a scalar multiple of $(1,-1)$,
then $G$ admits the following possibilities and nothing else: (i) infinitely many degenerate steady states, (ii) exactly one  nondegenerate positive steady state,  (iii) no positive steady states. 
	\item If $\beta_1 \beta_2 < 0$ and $\widetilde y - y$ is {\em not} a scalar multiple of $(1,-1)$,
then $G$ admits all the following possibilities and nothing else: (i) exactly two positive nondegenerate steady states, exactly one of which is exponentially stable, (ii) one doubly degenerate positive steady state,  (iii) no positive steady states. 
	\enen
\enen
Hence, (1) $G$ is {\em not} multistable, 
and also (2) $G$ is nondegenerately multistationary 
(in fact, $\capNPSS(G)=2$) if and only if  $y'-y =- \lambda (\widetilde y' - \widetilde y)$ for some $\lambda>0$, $\beta_1 \beta_2 < 0$, and  $\widetilde y - y$ is {\em not} a scalar multiple of $(1,-1)$.  
\end{theorem}

\begin{remark} \label{rmk:connect-2-thm}
We will see that the condition ``the slope of the reactant polytope is {\em not} $-1$'' (in Theorem~\ref{thm:s=2-r=2}) is equivalent to ``$\widetilde y - y$ is {\em not} a scalar multiple of $(1,-1)$'' (in Theorem~\ref{thm:s=2-r=2-alternate}).  We employ both, because the former theorem highlights geometric conditions on the box diagram, while the latter presents the corresponding algebraic conditions on the reactions.  Similarly, we will see that the box-diagram condition~\eqref{eq:boxes} is equivalent to the inequality $\beta_1 \beta_2 < 0$.
\end{remark}

\begin{example}[A reactant polytope with slope $-1$] \label{ex:-1-slope}
We revisit the network $G'~=~\{A+B \to 0 ,~ 2A \to 3A+B\}$ from Example~\ref{ex:cap}.
Here, $(\beta_1,\beta_2):=(-1,-1)*(1,-1)=(-1,1)$ and $\wt y - y=(1,-1)$, so by Theorem~\ref{thm:s=2-r=2-alternate} (part 3b), $G'$ admits infinitely many degenerate positive steady states (so, $\capPSS(G')=\infty$), but at most 1 nondegenerate one (so, $\capNPSS(G') = 1$).  We found the same result earlier.  Alternatively, the reactant polytope is the convex hull of $(1,1)$ and $(2,0)$, so its slope is $-1$, and thus, $G'$ is not nondegenerately multistationary (Theorem~\ref{thm:s=2-r=2}).
\end{example}

As in Corollary~\ref{cor:s=1-r=1-atoms} earlier, we find an infinite family of minimal multistationary networks:

\begin{corollary}[Classification of embedding-minimal multistationary/multistable networks with two reactions and two species] \label{cor:s=2-r=2-atoms} %
~
\been
\item A reaction network $G$ with exactly two species and two reactions is an embedding-minimal multistationary network if and only if $G$ is nondegenerately multistationary.  In other words, the 2-reaction, 2-species embedding-minimal multistationary networks are precisely those networks described in Theorem~\ref{thm:s=2-r=2}.
\item There are {\em no} 2-reaction, 2-species embedding-minimal multistable networks.
\enen
\end{corollary}

\begin{proof}
Part~2 is immediate from Theorem~\ref{thm:s=2-r=2-alternate}.  
For part~1, the $\Leftarrow$ direction is by definition.  For the $\Rightarrow$ direction, we need only show that for a 2-species, 2-reaction network, removing one species and/or one reaction yields a network without multiple nondegenerate positive steady states (removing more than one species or reaction would yield a trivial network).  First, a 1-reaction network is not multistationary (Theorem~\ref{thm:s=1-r=1}), so we need only consider removing exactly one species.  A resulting such network has one species and at most two reactions, so it is {\em not} nondegenerately multistationary (Theorem~\ref{thm:s=2-r=2}).  
\end{proof}

The next result, which is analogous to Corollary~\ref{cor:at-most-bi-s=1-r=1} earlier and will be superseded later by Corollary~\ref{cor:at-most-bi-r=2}, follows from the fact that the box diagrams~\eqref{eq:boxes} in Theorem~\ref{thm:s=2-r=2} do {\em not} fit in the triangle $\{(a,b) \mid 0 \leq a,~0 \leq b, ~a+b \leq 2 \}$. 
\begin{corollary} \label{cor:at-most-bi-s=2-r=2}
 Suppose that $G$ is an at-most-bimolecular reaction network. If $G$ has exactly two reactions and two species, then $G$ is {\em not} nondegenerately multistationary, i.e., $\capNPSS(G) \le 1$, and thus is {\em not} multistable. 
 \end{corollary}

Finally, we prove Proposition~\ref{prop:s+r} from earlier, which states that there exist nondegenerately multistationary reaction networks with $r$ reactions and $s$ species whenever $r+s \geq 4$, with $r \geq 2$.

\begin{proof}[Proof of Proposition~\ref{prop:s+r}]
If $s=1$, then $ r\geq 3$, so an $(r-1)$-alternating network is a 1-species, $r$-reaction network that is nondegenerately multistationary (Theorem~\ref{thm:s=1-r=1}).  So, now assume $s \geq 2$.  We construct an $s$-species, $r$-reaction network as follows. Begin with a $2$-species, $2$-reaction network for which the 
box diagram is one of the four depicted in~\eqref{eq:boxes}.  This network is nondegenerately multistationary by Theorem~\ref{thm:s=2-r=2}. Next, add $(r-2)$ extra reactions that each a scalar multiple of the original two reaction vectors, thereby preserving the (1-dimensional) stoichiometric subspace.  The resulting network is nondegenerately multistationary by Lemma~\ref{lem:lift}.  Finally, add in the remaining $(s-2)$ species trivially; for instance, add a species to both sides of a reaction (e.g., replace $A \to B$ by $A+C \to B+C$).
\end{proof}


\subsection{Proofs of Theorems~\ref{thm:s=2-r=2} and~\ref{thm:s=2-r=2-alternate}} \label{sec:proofs-s=r=2}

We begin with a lemma that interprets the $\beta_i$'s in terms of the box diagram.

\begin{lemma} \label{lem:2-rxn-mss}
Let $G$ be a network with exactly two species and two reactions, denoted by $y \to y'$ and $\widetilde y \to \widetilde y'$.   
  For each species $i=1,2$, let $\beta_i:=\beta_i(y \to y',~\widetilde y \to \widetilde y')$.  Let $\gamma$ denote the slope of the reaction vector $y \to y'$, and let $\alpha$ be the slope of the reactant polytope of the network.  Then, 
	$\beta_2/\beta_1 = \gamma \alpha$, 
whenever $\beta_1$ and $\beta_2$ are nonzero.
\end{lemma}

\begin{proof}
This is immediate from the fact that $\gamma=\frac{y_2'-y_2}{y_1'-y_1}$ and $\alpha=\frac{\widetilde y_2-y_2}{\widetilde y_1 - y_1}$. 
\end{proof}

\begin{proof}[Proof of Theorem~\ref{thm:s=2-r=2-alternate}]
Recall that the reaction vectors must be negative scalar multiples of each other for the network to be multistationary (Lemma~\ref{lem:2-rxn-dep}).
Now we use this relationship between the reaction vectors to simplify the differential equations:
\begin{align} \label{eqn:diff-eqns-pf}
\frac{dx_1}{dt}	&~=~ k(y'_1 -y_1)x_1^{y_1}x_2^{y_2} + \widetilde k(\widetilde y'_1- \widetilde y_1)x_1^{\widetilde y_1}x_2^{\widetilde y_2}
				~=~ k(y'_1 -y_1)x_1^{y_1}x_2^{y_2} - \lambda \widetilde k( y'_1 -  y_1)x_1^{\widetilde y_1}x_2^{\widetilde y_2}  \\
\label{eqn:diff-eqns-pf-2}
\frac{dx_2}{dt}	&~=~ k(y'_2 -y_2)x_1^{y_1}x_2^{y_2} + \widetilde k(\widetilde y'_2 - \widetilde y_2)x_1^{\widetilde y_1}x_2^{\widetilde y_2}
				~=~ k(y'_2 -y_2)x_1^{y_1}x_2^{y_2} - \lambda \widetilde k( y'_2 -  y_2)x_1^{\widetilde y_1}x_2^{\widetilde y_2} ~.
\end{align}
Note that $(y_2'-y_2) \frac{dx_1}{dt}=(y_1'-y_1) \frac{dx_2}{dt}$, so each stoichiometric compatibility class is a line: 
	\begin{align} \label{eq:stoic-eqn}
	(y_2'-y_2) x_1 +T ~=~(y_1'-y_1) x_2 ~, 
	\end{align}
for some $T \in \mathbb{R}$ for which the line~\eqref{eq:stoic-eqn} intersects the positive orthant of the $(x_1,x_2)$ plane. 
Also, note that it is impossible for $y'_1 - y_1=y'_2 - y_2=0$, i.e., at least one of differential equations~(\ref{eqn:diff-eqns-pf}--\ref{eqn:diff-eqns-pf-2}) must be nonzero, because otherwise the reaction $y \to y'$ would be trivial. Therefore, the steady state equations $\frac{dx_i}{dt}=0$ are equivalent to the single generalized-binomial equation:
\begin{align} \label{eqn:steady-state-eqn}
 k x_2^{y_2- \widetilde y_2} &~=~ \lambda \widetilde k x_1^{\widetilde y_1-y_1} ~.
\end{align}

We begin with the case when $(\beta_1,\beta_2) = \left( (y'_1 - y_1) (\widetilde y_1 -  y_1),~ (y'_2 - y_2)(\widetilde y_2 -  y_2) \right)=(0,0)$. 
We recall that $y'_1 - y_1=y'_2 - y_2=0$ is forbidden, so there are three subcases:
	\begin{enumerate}[(i)]
	\item $\widetilde y_1 - y_1=\widetilde y_2 -  y_2=0$ (in other words, the two reactants are equal: $y=\widetilde y$), 
	\item $y'_1 - y_1=\widetilde y_2 -  y_2=0$, and 
	\item $\widetilde y_1 -  y_1=y'_2 - y_2=0$.
	\end{enumerate}

In case (i), if $k=\lambda \widetilde k$, then equation~(\ref{eqn:steady-state-eqn}) holds trivially. So, in this case, every stoichiometric compatibility class entirely consists of (infinitely many) degenerate steady states.  On the other hand, if $k \neq \lambda \widetilde k$, then equation~(\ref{eqn:steady-state-eqn})
does not hold, so, in this case, there are no positive steady states.  Therefore, a network $G$ in case (i) satisfies the statement of the theorem.

For case (ii), we may assume that $\widetilde y_1 - y_1 \neq 0$ (otherwise, we are back in case (i)).  Then, the steady state equation~\eqref{eqn:steady-state-eqn} is a vertical line, namely, $x_1=\left(\frac{k}{\lambda \widetilde k}\right)^{1/(y_1 - \widetilde y_1)}$.  The stoichiometric compatibility class~\eqref{eq:stoic-eqn} also is a vertical line, so for one choice of $T$, these two lines coincide (resulting in a stoichiometric compatibility class consisting entirely of degenerate steady states); and for all other values of $T$, there are no positive steady states.  Hence, a network $G$ in case~(ii) satisfies the statement of the theorem.
Finally, case (iii) is symmetric to case (ii).

Now assume that $\beta_1= (y'_1 - y_1) (\widetilde y_1 -  y_1)=0$, while $\beta_2 = (y'_2 - y_2)(\widetilde y_2 -  y_2) \neq 0$.  First, if $y'_1 - y_1=0$, then the compatibility class~\eqref{eq:stoic-eqn} is a vertical line passing through the positive orthant of the $(x_1,x_2)$ plane (here we use the hypothesis $y_2' \neq y_2$), and the steady state equation~\eqref{eqn:steady-state-eqn} is a positive function that defines $x_2$ in terms of $x_1$ (here, we use $\widetilde y_2 \neq y_2$).  Thus, there is a unique  (nondegenerate) positive steady state.  In the remaining case, when $y'_1 - y_1 \neq 0$ and $\widetilde y_1 -  y_1=0$, the compatibility class~\eqref{eq:stoic-eqn} is a non-horizontal and non-vertical line  (because $y_i' \neq y_i$ for $i=1,2$) with $x_2$-intercept $-T/(y_2'-y_2)$ that passes through the positive orthant, and the steady-state equation~\eqref{eqn:steady-state-eqn} defines a horizontal line (again, we use $\widetilde y_2 \neq y_2$).  So, for some choices of (negative) $T$, there is a unique (nondegenerate) positive steady state, and for others, there is no positive steady state.  Finally, the case when $\beta_1 =0$ and $\beta_2 \neq 0$ is symmetric.

Now assume that $\beta_1 \neq 0$ and $\beta_2 \neq 0$.  Then, by Lemma~\ref{lem:2-rxn-mss}, 
$\beta_2/\beta_1 = \gamma \alpha \neq 0$,
where $\gamma:=(y_2'-y_2)/(y_1'-y_1)$ is the (nonzero) slope of the reaction vector $y \to y'$ (and thus also that of $\widetilde y \to \widetilde y'$), and 
$\alpha:=(\widetilde y_2-y_2)/(\widetilde y_1 - y_1)$ is the (nonzero) slope of the reactant polytope.

Via $K:=\left( {\lambda \widetilde k}/ {k}\right)^{1/(y_2-\widetilde y_2)}$, we rewrite the stoichiometric and steady state equations~(\ref{eq:stoic-eqn}--\ref{eqn:steady-state-eqn}): 
	\begin{align*}
	x_2 &= g(x_1) ~:=~ \gamma x_1 + {T}/({y_1'-y_1}) \\
	x_2 &= h(x_1) ~:=~ K x_1^	 {-1/\alpha}   ~, 	
	\end{align*}
so the positive steady states correspond to positive solutions $x_1\in \mathbb{R}$ to the equation $h(x_1)-g(x_1)=0$.  Accordingly, the remainder of the proof comprises counting these solutions; more precisely, for fixed values of $\gamma$ and $\alpha$, we will determine the possible numbers of solutions as $K>0$ and $T \in \mathbb{R}$ vary (note that $K$ can take any positive value, by choice of $k$ and $\widetilde k$).

We begin with the subcase  $\beta_1 \beta_2 >0$, i.e., $\gamma \alpha >0$.  If $\alpha>0$ and $\gamma>0$, then $h(x_1)- g(x_1)$ is a decreasing function for $x_1>0$ (because $h$ and $-g$ are decreasing) that is positive for positive $x_1 \approx 0$ and negative for $x_1 \gg 0$; hence, there is a unique positive steady state.  Similarly,  if $\alpha<0$ and $\gamma<0$, then $h(x_1)- g(x_1)$ is increasing (because $h$ and $-g$ are increasing) and is negative for positive $x_1 \approx 0$ and positive for $x_1 \gg 0$; hence, there is a unique positive steady state, and it is nondegenerate.  

Now consider the subcase when the slope of the reactant polytope is $\alpha =-1$ (i.e., $\wt y - y $ is a scalar multiple of $(1,-1)$, as the reactant polytope is the line segment from $y$ to $\wt y$) and the slope of the reaction vectors is $\gamma>0$.  Both $g$ and $h$ are lines with positive slopes, namely $\gamma$ and $K$, which pass through the positive orthant.  Therefore, there can be either infinitely many (degenerate) positive steady states (for instance, when $T=0$ and $K=\gamma$), a unique (nondegenerate) positive steady state (e.g., if $T>0$ and $K>\gamma$), or no positive steady states (e.g., if $T>0$ and $K<\gamma$).

The remaining subcase is when the slope $\alpha$ of the reactant polytope is not $-1$ and its sign is opposite that of the (shared, nonzero) slope of the reaction vectors.  We have two further subcases, based on the signs of $\gamma$ and $\alpha$.  First, if $\gamma <0$ (so, $\alpha >0$), then $g$ is linear and decreasing, and $h$ is concave down and decreasing, so there are at most 2 intersection points of their graphs (i.e., up to 2 positive steady states).  We will show that 0, 1, and 2 intersection points are possible.  
For any choice of $h$ (which is concave down and the derivative takes all values in $\mathbb{R}_{<0}$), take the (unique) tangent line to the graph of $h$ with slope $\gamma$: this line has the form of $g$, and the resulting system has a doubly degenerate positive steady state.  Now, shift the line slightly to the left (now it does not intersect the graph of $h$), so that it still intersects the positive orthant.  It follows that the equation of the resulting line has the form of $g$, and there are no positive steady states for this $g$ and $h$.  Similarly, shifting the line slightly to the right yields 2 (nondegenerate) positive steady states -- one exponentially stable, one unstable.

Finally, the $\gamma >0$ (so, $\alpha < 0$) case is similar to the $\gamma<0$ case.  Here, $g$ is a line with positive slope, and $h$ is a positive (for $x_1>0$) increasing function that is either concave up with derivative taking all possible positive values (\caseOne) or concave down with derivative taking all possible negative values (\caseTwo).  Thus, again there are at most 2 intersection points.  Now fix $h$, and then pick $T$ such that the graph of $g$ is a tangent line of the graph of $h$: this corresponds to a doubly degenerate point.  Accordingly, adjusting $T$ slightly so the graph of $g$ moves to the left yields either 2 nondegenerate positive steady states -- one exponentially stable, one unstable (\caseOne) or 0 positive steady states (\caseTwo).  Similarly, moving $g$ slightly to the right yields either 0 or 2 (nondegenerate) positive steady states, respectively.  
%
\end{proof}

 \begin{proof}[Proof of Theorem~\ref{thm:s=2-r=2}]
This result follows directly from Theorem~\ref{thm:s=2-r=2-alternate}, once we reinterpret the conditions characterizing nondegenerate multistationarity (namely, $\beta_1 \beta_2 < 0$ and $\widetilde y - y$ is {\em not} a scalar multiple of $(1,-1)$) in terms of box diagrams.  To this end, the reactant polytope is the line segment from $y$ to $\wt y$, so $\wt y - y$ is {\em not} a scalar multiple of $(1,-1)$ if and only if the reactant polytope has slope {\em not} equal to $-1$.  
Also, by Lemma~\ref{lem:2-rxn-mss}, $\beta_1 \beta_2 <0$ if and only if (a) the reactants $y$ and $\widetilde y$ differ in both coordinates and (b) 
the slopes of the reaction vector of $y \to y'$ and the slope of the reactant polytope are both defined and nonzero, and have opposite signs;
equivalently, the box diagram of the network has one of the four forms depicted in~\eqref{eq:boxes}.
 \end{proof}

\section{Networks with two (possibly reversible) reactions} \label{sec:two-rxn-net}
The previous section analyzed networks with exactly two reactions and two species.  
Now we again consider networks with two reactions, but we allow any number of species and permit each reaction to be possibly reversible.  
In this setting, Boros used Feinberg's deficiency one algorithm\footnote{The original presentation and proof of the deficiency one algorithm are in~\cite{feinberg1988chemical,Fein95DefOne}.  A concise description is in~\cite{joshi2013complete}.} to resolve the question of multistationarity (Question~\ref{q:main}) for those networks  that are ``regular'' (a technical condition that is stronger than being consistent, and is required to apply the deficiency one algorithm)~\cite{revisiting}.  We describe those results in Theorems~\ref{thm:r=2},~\ref{thm:i=1-r=1}, and~\ref{thm:rev=2}.  Moreover, we go further: we determine which of these networks admit multiple nondegenerate/stable steady states (to address Questions~\ref{q:main-nondeg} and~\ref{q:main-multistable}).

\subsection{Networks with two reactions} \label{subsec:2-irr}

Here we consider networks with exactly two reactions $y \to y'$ and $\widetilde y \to \widetilde y'$.  This includes the case of a single reversible reaction $y \rla y'$. 
As an aside, note that Boros analyzed the case of a single reversible reaction together with all outflows~\cite{revisiting}; here, we do not analyze such networks.

\subsubsection{Main results} \label{subsubsec:2-irr-main}

We now state the two main results of this subsection, which generalize Theorems~\ref{thm:s=2-r=2} and~\ref{thm:s=2-r=2-alternate}.   Theorem~\ref{thm:r=2} in the 2-linkage-class case was done by Boros~\cite{revisiting}\footnote{Boros did not use the function $\beta$, but it is easy to check that the conditions he gave are equivalent to those here.}.  We give an alternate proof, which we postpone until Section~\ref{subsubsec:2-irr-pf}.
\begin{theorem}[Classification of multistationary networks with two reactions] \label{thm:r=2}
Let $G$ be a network that consists of two reactions, denoted by $y \to y'$ and $\widetilde y \to \widetilde y'$.  Let $\beta := (y'-y) \ast (\wt y -y) \in \Z^s$, where $s$ is the number of species. 
Then $G$ is multistationary (i.e., $\capPSS(G) \geq 2$) if and only if 
\been
\item  the reaction vectors are negative scalar multiples of each other: $y'-y =- \lambda (\widetilde y' - \widetilde y)$ for some $\lambda>0$, and 
\item 
$\beta$ either is the zero vector or has at least one positive and at least one negative coordinate.
\enen
\end{theorem}

\begin{theorem}[Classification of {\em nondegenerately} multistationary networks with two reactions] \label{thm:speculation}
Let $G$ be a network that consists of two reactions, denoted by $y \to y'$ and $\widetilde y \to \widetilde y'$.  Let $\beta := (y'-y) \ast (\wt y -y)$. 
 Then $G$ is nondegenerately multistationary if and only if 
 \been
\item  the reaction vectors are negative scalar multiples of each other: $y'-y =- \lambda (\widetilde y' - \widetilde y)$ for some $\lambda>0$, 
\item $\beta$ 
has at least one positive and at least one negative coordinate, and 
\item if $\beta$ has exactly one positive and one negative coordinate, say $\beta_i>0$ and $\beta_j<0$, then the projection of the reactant polytope to the $(i,j)$ plane does {\em  not} have slope $-1$. 
\enen
 \end{theorem}

\begin{remark} \label{rmk:2-irr-box}
Condition~2 of Theorem~\ref{thm:speculation} is equivalent to requiring that for some choice of two species $i$ and $j$, the projection of the box diagram to the $(i,j)$-plane has one of the four (zigzag) forms~\eqref{eq:boxes} from Theorem~\ref{thm:s=2-r=2}. To see this, recall Remark~\ref{rmk:connect-2-thm}.
\end{remark}

\begin{remark} \label{rmk:multistab}
We wish to extend Theorem~\ref{thm:speculation} to include a criterion for multistability.  However, at present, we do not know precisely which 2-reaction networks are multistable.  In the 1- and 2-species cases, none are multistable (Theorems~\ref{thm:s=1-r=1} and~\ref{thm:s=2-r=2-alternate}).
In the 3-species case, some are multistable: 
$\{A \to B+C,~ 2A+B+C \to 3A \}$ is, but not 
$\{A+C \to B,~ 2A+B \to 3A+C \}$.
\end{remark}

 The following result is analogous to Corollary~\ref{cor:at-most-bi-s=1-r=1}:
\begin{corollary} \label{cor:at-most-bi-r=2}
 Suppose that $G$ is an at-most-bimolecular reaction network. If $G$ has exactly two reactions, then $G$ is {\em not} nondegenerately multistationary, i.e., $\capNPSS(G) \le 1$, and thus is {\em not} multistable. 
 \end{corollary}
 \begin{proof}
 Let $G$ be a 2-reaction, nondegenerately multistationary network.  By Theorem~\ref{thm:speculation}, $\beta_i:=(y_i'-y_i)(\wt y_i - y_i)<0$ for some $i$, and also $y_i'-y_i = - \lambda ( \wt y_i'-\wt y_i)$.  So, either $\wt y_i' < \wt y_i < y_i < y_i'$ or $y_i' < y_i < \wt y_i < \wt y_i'$, and thus $y_i' \geq 3$ or $\wt y_i' \geq 3$. Hence, $G$ is {\em not} at-most-bimolecular.
 \end{proof}

Here we extend Corollary~\ref{cor:s=2-r=2-atoms}, which classified multistationary networks with 2 reactions and 2 species, to allow for any number of species.
\begin{corollary}[Classification of embedding-minimal multistationary networks with two reactions] \label{cor:r=2-atoms} %
~
A reaction network $G$ with exactly two reactions, denoted by $y \to y'$ and $\widetilde y \to \widetilde y'$,
 is an embedding-minimal multistationary network if and only if either
\begin{enumerate}[(1)]
	\item $G$ is a 2-species network described in Theorem~\ref{thm:s=2-r=2},
or 
	\item $G$ is a 3-species network for which 
 $\beta := (y'-y) \ast (\wt y -y) \in \mathbb{Z}^3$ contains either 2 positive and 1 negative or 1 negative and 2 positive coordinates,
and for the two pairs $\{i,j\}$ for which $\beta_i \beta_j <0$, the projection of the reactant polytope to the $(i,j)$ plane has slope $-1$. 
\end{enumerate}
\end{corollary}

\begin{proof}
The networks in (1) are embedding-minimal multistationary networks (Corollary~\ref{cor:s=2-r=2-atoms}).  The networks in (2), we now claim, are also embedding-minimal multistationary networks.  First, removing a reaction from such a network  yields a 1-reaction network, which is not multistationary (Theorem~\ref{thm:s=1-r=1}).  Also, removing a species yields a 2-reaction, 2-species network such that either $\beta$ has coordinates of only one sign or $\beta$ has one positive and one negative coordinate but the reactant polytope has slope $-1$; such a network is not nondegenerately multistationary (Theorem~\ref{thm:s=2-r=2-alternate}).  Thus, the $\Rightarrow$ implication of the corollary holds.

To prove $\Leftarrow$, we must show that, besides those in (1) and (2), no other 2-reaction networks are embedding-minimal multistationary networks.  
First, there are none with only 1 or 2 species (Theorems~\ref{thm:s=1-r=1} and~\ref{thm:s=2-r=2-alternate}).
Next, consider a nondegenerately multstationary 2-reaction network with exactly 3 species.  If $\beta$ contains a zero coordinate, then that species can be removed: doing this does not affect whether the conditions for multistationarity in Theorem~\ref{thm:speculation} hold.
Also, for any pair $\{i,j\}$ of species with $\beta_i \beta_j <0$, then if the projection of the reactant polytope to the $(i,j)$ plane has slope {\em not} equal to $-1$, then the remaining third species can be removed, yielding again a nondegenerately multistationary network (Theorem~\ref{thm:s=2-r=2-alternate}).  
Finally, for 2-reaction networks with 4 or more species, we can remove either a species corresponding to a 0 coordinate of $\beta$, or a species corrresponding to a coordinate of $\beta$ whose sign appears more than once in $\beta$: again, this does not affect whether the conditions for multistationarity in Theorem~\ref{thm:speculation} hold.
\end{proof}

\subsubsection{Proof of Theorems~\ref{thm:r=2} and \ref{thm:speculation}} \label{subsubsec:2-irr-pf}
Our proof of Theorems~\ref{thm:r=2} and \ref{thm:speculation} uses Theorem~\ref{thm:s=2-r=2-alternate} (which Theorem~\ref{thm:speculation} generalizes).

\begin{proof}[Proof of Theorems~\ref{thm:r=2} and \ref{thm:speculation}]
Condition~1 (which is the same in both theorems) must hold for the network $G$ to be multistationary (Lemma~\ref{lem:2-rxn-dep}), so we will assume that it holds for the rest of the proof.  Thus, we can write the differential equation for species $i=1,2, \ldots, s$ as:
\begin{align} \label{eq:ODE-2-rxn}
\frac{dx_i}{dt} ~=~
k(y_i'-y_i) x^y + \wt k (\wt y_i'-\wt y_i) x^{\wt y}
~=~
(y_i'-y_i) \left( k x^y - \frac{\wt k}{\lambda} x^{\wt y} \right)~.
\end{align}
For any species $i$ for which $y_i' = y_i$, that species concentration is fixed for all time, and so we may remove species $i$ from further consideration. In other words, replacing $G$ by the resulting embedded network does {\em not} affect whether it is multistationary or nondegenerately multistationary, and also does {\em not} affect the conditions listed in either theorem.  So, for the rest of the proof, we will assume that $\beta_j=(y_j'-y_j)(\widetilde y_j - y_j)=0$ if and only if $\wt y_j=y_j$.

We begin with the case when $\beta$ is the zero vector ($\beta_j=0$ for all $j$), or, equivalently, the two reactants are equal: $\wt y=y$.    Thus, from equation~\eqref{eq:ODE-2-rxn}, the differential equations are: 
\begin{align}
	\frac{dx_i}{dt}~=	 (y_i' - y_i) (k-\widetilde k/\lambda) x^y ~.
\end{align} 
Hence, if $\lambda k \neq \widetilde k$, then there are {\em no} positive steady states (in any stoichiometric compatibility class), and
if  $\lambda k = \widetilde k$, then every stoichiometric compatibility class consists of (infinitely many) degenerate steady states.  Thus,  $\capPSS(G)=\infty$ and $\capNPSS(G)=0$.

Having completed the case when all $\beta_i$'s are zero, we will assume that at least one $\beta_i$ is nonzero for the remainder of the proof. In fact, we will soon reduce to the case that every $\beta_i$ is nonzero.
To do so, we use the differential equations~\eqref{eq:ODE-2-rxn} to write the steady-state equation as the following generalized-monomial equality:
\begin{align} \label{eq:ODE-2-rxn-rewritten}
x_1^{y_1-\wt y_1}
x_2^{y_2-\wt y_2}
\cdots
x_s^{y_s-\wt y_s}
~=:~
x^{y- \wt y} ~=~ \frac{\wt k}{\lambda k} ~=: \kappa~.
\end{align}
Also from~\eqref{eq:ODE-2-rxn}, the conservation laws are, for $i=2,3,\dots, s$:
\begin{align*} 
\frac{d x_i}{dt} ~=~ \left ( \frac{y_j'-y_j}{y_1'-y_1} \right) \frac{d x_1}{dt} ~=:~ \gamma_i \frac{dx_1}{dt}~,
\end{align*}
where $\gamma_i:=\frac{y_j'-y_j}{y_1'-y_1} $ is the slope of the projection of the reaction vector to the $(1,i)$ plane.
Equivalently, for $i=2,3,\dots, s$:
\begin{align} \label{eq:conservation}
x_i= c_i + \gamma_i x_1~.
\end{align}
Combining~\eqref{eq:ODE-2-rxn-rewritten} and~\eqref{eq:conservation}, we obtain:
\begin{align} \label{eq:steady-state-1-variable}
x_1^{y_1-\wt y_1}
(c_2 + \gamma_2 x_1)^{y_2-\wt y_2}
\cdots
(c_s + \gamma_s x_1)^{y_s-\wt y_s}
~=~ \kappa~.
\end{align}
Thus, 
our goal is to determine whether (for some values of $\kappa>0$ and $c_i \in \mathbb{R}$) the one-variable equation~\eqref{eq:steady-state-1-variable} admits multiple (nondegenerate or degenerate) roots $x_1 \in \mathbb{R}$ that lie in the (possibly empty) interval 
$D=D(\{c_i\}):=(d_1,d_2)$, where
$d_1:= \max \left( \{0\} \cup  \{ -c_i/\gamma_i \mid 2 \leq i \leq s {\rm ~ and ~}\gamma_i >0 \} \right) $ and 
$d_2:= \min \left( \{\infty\} \cup \{- c_i/\gamma_i \mid 2 \leq i \leq s {\rm ~ and ~}\gamma_i <0\} \right)$.

Now we reduce to the case when each $\beta_i$ is nonzero.  To do so, we may assume that species $i=1,2,\dots,s'$ are precisely the ones that satisfy $\beta_i \neq 0$.  (So, $\beta_j=0$, i.e., $\wt y_j=y_j$, for the remaining species $j=s'+1,2,\dots, s$.)  Thus, in this case, equation~\eqref{eq:steady-state-1-variable} reduces to:
\begin{align} \label{eq:steady-state-1-variable-ClaimB}
x_1^{y_1-\wt y_1}
(c_2 + \gamma_2 x_1)^{y_2-\wt y_2} \cdots
(c_{s'} + \gamma_{s'} x_1)^{y_{s'}-\wt y_{s'}}
~=~ \kappa~.
\end{align}
Now let $G'$ denote the embedded network of $G$ obtained by removing species $s'+1,s'+2, \dots, s$.  We claim now that $G'$ and $G$ have the same capacities for multistationarity:
	\begin{align} \label{eq:subclaim}
	\capPSS(G)=\capPSS(G') \quad {\rm and} \quad 
	\capNPSS(G)=\capNPSS(G')~.
	\end{align}
To prove~\eqref{eq:subclaim}, note that for $G'$, we are counting the maximum number of (nondegenerate or degenerate) roots $x_1 \in \mathbb{R}$ of the equation~\eqref{eq:steady-state-1-variable-ClaimB} that lie in the interval 
$D' :=(d_1',d_2')$, where
$d_1':= \max \left( \{0\} \cup  \{ -c_i/\gamma_i \mid 2 \leq i \leq s' {\rm ~ and ~}\gamma_i >0 \} \right) $ and 
$d_2':= \min \left( \{\infty\} \cup \{- c_i/\gamma_i \mid 2 \leq i \leq s' {\rm ~ and ~}\gamma_i <0\} \right)$ 
 (so, $D'$ contains $D$).  
Thus, the ``$\leq$'' inequalities in~\eqref{eq:subclaim} hold.  

For the ``$\geq$'' inequalities, the steady-state equation~\eqref{eq:steady-state-1-variable-ClaimB} admits degenerate steady states if and only if the equation itself is degenerate (if the left-hand side is itself a constant and equal to $\kappa$) -- indeed, in this case the $c_i$'s can be picked so that the intervals $D$ and $D'$ are nonempty -- so both $G$ and $G'$ would admit infinitely many degenerate steady states.  
In the nondegenerate case, if equation~\eqref{eq:steady-state-1-variable} admits nondegenerate steady states $x_1(1), x_1(2), \dots, x_1(l)$ in the interval $D'$ (for some choice of the $c_1,c_2,\dots,c_{s'}$ and $\kappa$), then we may pick $c_{s'+1}, \dots, c_s \in \mathbb{R}$ so that 
all the steady states $x_1(j)$ are in $D$, as follows.  If $\gamma_i>0$, let $c_i:=0$; and if $\gamma_i<0$, pick $c_i$ such that 
$-c_i/\gamma_i > \max \{ x_1(1), x_1(2), \dots, x_1(l)\} $.
This proves ``$\geq$'' of~\eqref{eq:subclaim}, and thus for the rest of the proof we may assume that $\beta_i \neq 0$ for all species $i$.

Now we consider the cases with few species (assuming that $\beta_i \neq 0$ for all $i$).  First, the $s=1$ case is clear: 1-species networks with precisely two reactions are not multistationary (Theorem~\ref{thm:s=1-r=1}), and $\beta$ has only one coordinate (so no coordinates of opposite sign).  In the $s=2$ case, what we must show is that the 2-species, 2-reaction network $G$ is multistationary if and only if $\beta_1 \beta_2 <0$, and is nondegenerately multistationary if and only if $\beta_1 \beta_2 <0$ and the slope of the reactant polytope is not -1.  This is contained in Theorem~\ref{thm:s=2-r=2-alternate}; see also Remark~\ref{rmk:connect-2-thm}.  

So, for the remainder of the proof, assume that $s \geq 3$ (and $\beta_i \neq 0$ for all $i$).  What we must prove is that the following are equivalent: (i) $G$ is multistationary, (ii) $G$ is nondegenerately multistationary, and (iii) $\beta$ contains at least one positive and at least one negative coordinate.  Clearly, (ii) $\Rightarrow$ (i), so we will complete the proof by showing (i) $\Rightarrow$ (iii) $\Rightarrow$ (ii).

To this end, take the $1/(y_1 - \widetilde y_1)$-th root of~\eqref{eq:steady-state-1-variable-ClaimB}, and call $f$ the resulting function on $D$:
\begin{align} \label{eq:onedim}
f(x_1) ~:=~
  x_1 (c_2 + \gamma_2 x_1)^{ \frac{\wt y_2-y_2}{\wt y_1-y_1}} \ldots (c_s + \gamma_s x_1)^{ \frac{\wt y_s-y_s}{\wt y_1-y_1}}  
~=~ 
 x_1 (c_2 + \gamma_2 x_1)^{\alpha_2} \ldots (c_s + \gamma_s x_1)^{\alpha_s}, 
\end{align}
where $\alpha_i := \frac{\wt y_i-y_i}{\wt y_1-y_1}$ for $i=2,3,\dots,s$.  Note that by construction, $\beta_i/\beta_1 = \gamma_i \alpha_i$ (cf. Lemma~\ref{lem:2-rxn-mss}).

Recall that our goal is to determine precisely when (for some values of $K>0$ and $c_i \in \mathbb{R}$) the equation $f(x_1)=K$ admits multiple roots $x_1 \in \mathbb{R}$ that lie in the subinterval $D$ of $\mathbb{R}_{>0}$ defined by $c_i+\gamma_i x_1>0$ for all $i$.  
As $f > 0$ on $D$, we can define the following function on $D$:
\begin{align*}
g(x_1) ~:=~ \log f(x_1) ~=~ \sum_{i=1}^s \alpha_i \log\left(c_i + \gamma_i x_1 \right) ~,\end{align*}
where $c_1 :=0$, $\alpha_1:=1$, and $\gamma_1 :=1$.  
%
Here we compute $g'$ and $g''$:
\begin{align}
g'(x_1) 
	&~=~ \sum_{i=1}^s \frac{\alpha_i  \gamma_i}{c_i + \gamma_i x_1}
	~=~\langle \alpha, (\gamma * \rho) \rangle~=~ \langle \alpha * \gamma,~\rho \rangle 
	~=~ \frac{1}{\beta_1} \langle \beta , ~\rho \rangle ~
	\label{eq:g'}
 \\
g''(x_1) 
	&~=~ - \sum_{i=1}^s \frac{\alpha_i  \gamma_i^2}{(c_i + \gamma_i x_1)^2}
	~=~ - \langle \alpha, (\gamma * \rho)* (\gamma * \rho) \rangle~,
	\label{eq:g''}
\end{align} 
where we have introduced a change of coordinates: 
\begin{align} \label{eq:rho}
\rho_1:=1/x_1 \quad {\rm and } \quad
\rho_i:= 1/(c_i+\gamma_i x_1)~,
\end{align}
for $i=2,3,\dots,s$ (so, $\rho \in \mathbb{R}^s_{>0}$ whenever $x_1 \in D$).


Now, if $f=K$ has a multiple root in $D$, then $f'$ has a zero in $D$, and thus $g'$ does too (because $\log$ is monotonic).  So, by~\eqref{eq:g'}, $\langle \beta, ~\rho \rangle=0$
for some $\rho \in \mathbb{R}^s_{+}$, and thus $\beta$ has at least one positive and at least one negative coordinate.  Hence, (i) $\Rightarrow$ (iii).

For (iii) $\Rightarrow$ (ii), assume that $\beta$ contains both positive and negative coordinates.  Thus, there exists $\widetilde \rho \in \mathbb{R}^s_{>0}$ such that $\langle B,~\widetilde\rho \rangle=0$.  So, from~\eqref{eq:g'}, 
\begin{align} \label{eq:in-prod-1}
\langle \alpha, (\gamma * \widetilde \rho) \rangle~=~0~.
\end{align}
We claim that it suffices to show that 
\begin{align} \label{eq:in-prod-2}
\langle \alpha, (\gamma *\widetilde \rho)* (\gamma *\widetilde \rho) \rangle ~\neq~ 0
\end{align}
for this $\widetilde\rho$ or another choice of $\widetilde \rho \in \mathbb{R}^s_{>0}$ for which equation~\eqref{eq:in-prod-1} also holds.  This is because, via the change of coordinates~\eqref{eq:rho} from $\widetilde \rho$ to $\widetilde x_1$ and $\widetilde c_i$'s, this says that $g'(\widetilde x_1)=0$ and $g''(\widetilde x_1) \neq 0$ (recall~\eqref{eq:g'} and~\eqref{eq:g'}).  
Thus, $g$, and thus $f$ as well, has a local extremum at $\widetilde x_1$.  If it is a local minimum, then the equation $f(x_1) =f(\widetilde x_1)+ \epsilon$ (for $\epsilon>0$ sufficiently small) has multiple roots $x_1$ in $D$; if it is a local maximum, then $f(x_1)=f(\widetilde x_1)- \epsilon$ has multiple roots.
	
Thus, to complete the proof of (iii) $\Rightarrow$ (ii), we must show that there exists $\widetilde \rho \in \mathbb{R}^s_{>0}$ such that~\eqref{eq:in-prod-1} and~\eqref{eq:in-prod-2} hold.  Note that 
\begin{align*}
\left\{ \gamma * \widetilde \rho \mid \widetilde \rho \in \mathbb{R}^s_{>0} \right\} 
~=~
{\rm sign}^{-1} (\gamma)~.
\end{align*}
Therefore, we want to show that there exists $v \in {\rm sign}^{-1} (\gamma)$ such that $\langle \alpha, v \rangle=0$ and $\langle \alpha, v*v \rangle \neq 0$.  This is the content of Lemma~\ref{lem:square-lem} below, and this completes the proof.
\end{proof}

\begin{lemma} \label{lem:square-lem}
Let $\alpha, \gamma \in \R^s$, where $s \geq 3$. Assume that $\alpha_i \ne 0$ and $\gamma_i \ne 0$  for all $i=1,2,\dots, s$. 
Let $\CC_\gamma:={\rm sign}^{-1} (\gamma)$ denote the orthant-interior in $\mathbb{R}^s$ that contains $\gamma$, and assume that $\alpha^{\perp} \cap \CC_\gamma$ is nonempty.
Then there exists a vector $v = (v_1,v_2,\ldots, v_s) \in \alpha^{\perp} \cap \CC_\gamma$ such that $v*v := (v_1^2, v_2^2,\ldots, v_s^2) \notin \alpha^{\perp}$.
\end{lemma}
\begin{proof}
Take $w \in \alpha^{\perp}\cap \CC_\gamma$.  If $w * w \notin \alpha^{\perp}$, then we are done.  Otherwise, we have:
\begin{align} \label{eq:w-properties}
\langle \alpha, w \rangle~ = ~ \langle \alpha, w*w \rangle ~=~0~, 
\end{align}
with ${\rm sign}(w)={\rm sign}(\gamma)$. 
We add an error term to $w$, and call this sum $v$:
\begin{align*}  
v 
~:=~ w +(\epsilon, \dots, \ep, \sigma\epsilon, \epsilon') 
~=~ (w_1 + \epsilon, \ldots, w_{s-2} + \epsilon, w_{s-1} + \sigma \epsilon, w_s + \epsilon') ~.
\end{align*}
Note that ${\rm sign}(v)={\rm sign}(w)$ for sufficiently small $\ep$ and $\ep'$ (when $\sigma$ is fixed). 
Then:
\begin{align*}
v * v ~&=~ 
	w * w + 2 \ep w + 2\ep (\sigma-1) w_{s-1} e_{s-1}+ 2 (\ep' - \ep) w_s e_s \\
	& \quad \quad	+ \ep^2(1,1,\dots,1)+ \ep^2(\sigma^2-1)e_{s-1}+ (\ep'^2-\ep^2)e_s	~,
\end{align*}
where $e_i$ is the $i$-th canonical basis vector in $\mathbb{R}^s$. 
Let $\Gamma_q := \sum_{i=1}^q \alpha_i$. 
Thus, using~\eqref{eq:w-properties}, we have: 
\begin{align} \label{eq:a-v2}
\langle \alpha, v \rangle &~=~ \epsilon \Gamma_{s-2} + \sigma \epsilon \alpha_{s-1} + \epsilon' \alpha_s
\nonumber\\
\langle \alpha, v*v \rangle &~=~
	2\ep (\sigma-1)\alpha_{s-1} w_{s-1} + 2 (\ep' - \ep) \alpha_s w_s + \ep^2 \Gamma_s+ \ep^2(\sigma^2-1)\alpha_{s-1}+ (\ep'^2-\ep^2)\alpha_s	~.
\end{align}

We consider two cases, based on whether $\Gamma_s$ is zero.  First, 
assume $\Gamma_s \ne 0$. Let $\sigma=1$ and $\ep' = - \frac{\ep \Gamma_{s-1}}{\alpha_s}$.  By construction, the error term $(\epsilon, \dots, \ep, \sigma\epsilon, \epsilon')$ is in $\alpha^{\perp}$.  Thus, because $w\in \alpha^{\perp}$ and $v$ is the sum of $w$ and the error term, $v \in \alpha^{\perp}$, and thus $v \in \alpha^{\perp} \cap \CC_\gamma$ when $\ep$ is sufficiently small.  So, to complete the $\Gamma_s \ne 0$ case, we need only show that $\langle \alpha, v*v \rangle \neq 0$ for some choice of small $\ep$.  To this end, a straightforward computation that uses the properties~\eqref{eq:w-properties} of $w$ yields:
\[
\langle \alpha, v*v \rangle  ~=~
\sum_{i=1}^s \alpha_i v_i^2 ~=~ \frac{\ep \Gamma_s}{\alpha_s}(\ep \Gamma_{s-1} - 2 w_s \alpha_s)~.
\]
Thus, $ \langle \alpha, v*v \rangle \neq 0$ for sufficiently small $\ep \ne 0$ (because $\Gamma_s \neq 0$, $\alpha_s \neq 0$, and $\ep \Gamma_{s-1} - 2 w_s \alpha_s \approx -2w_s \alpha_s \neq 0$: recall that ${\rm sign}(w)={\rm sign}(\gamma)$, which has no zero coordinates).  

Now suppose that $\Gamma_s =0$. 
There exists a pair of indices $i$ and $j$ such that $\alpha_i  \neq -\alpha_j $ (because $s \ge 3$ and $\alpha$ is nonzero). Relabel the indices if necessary, so that $\alpha_{s-1} \neq -\alpha_s$.
 Let $\sigma=-1$ and $\ds \ep' = \ep \left( 1 + 2 \frac{\alpha_{s-1}}{\alpha_s} \right)$. 
 As in the previous case, by construction, $v \in \alpha^{\perp} \cap \CC_\gamma$ for sufficiently small $\ep$.  Then by~\eqref{eq:a-v2} and using $\Gamma_s =0$, we have the following straightforward calculation:
\begin{align}
\langle \alpha , v * v \rangle~&=~
	-4 \ep \alpha_{s-1} w_{s-1} + 2( \ep' - \ep) \alpha_s w_s +0+0+ 
		(\ep'^2-\ep^2)\alpha_s \notag \\
		&=~ 4 \ep \alpha_{s-1} \left[ 
w_s - w_{s-1} 
+ \ep \left(\frac{\alpha_{s-1} + \alpha_s}{\alpha_s} \right) \right]~. \label{eq:last}
\end{align}
Thus, $ \langle \alpha, v*v \rangle \neq 0$ for some small value of $\ep > 0$, because $\alpha_{s-1} \neq 0$ and $\alpha_{s-1} \neq -\alpha_s$ (so \eqref{eq:last} is not identically zero for all values of $\ep$).  Thus, the $\Gamma_s=0$ case is done.
\end{proof}


\subsection{Networks with one irreversible and one reversible reaction } \label{subsec:1-ir-1-rev}

Theorem~\ref{thm:i=1-r=1} below classifies the multistationary networks that consist of one reversible reaction and one irreversible reaction.  Earlier, we classified the 1-species networks of this form (Theorem~\ref{thm:s=1-r=1}).  
Such a network must be a 2-alternating network, which in this case means that the reactions $y \lra y'$ and $\widetilde y \to \widetilde y'$ satisfy either $\min\{y,y'\}<\max\{y,y'\}< \widetilde y < \widetilde y'$ (``$\lra \, \, \to$'', e.g., $A\lra 2A,~3A \to 4A$) or 
$\widetilde y' <
\widetilde y <
\min\{y,y'\}< \max\{y,y'\} $ (``$ \leftarrow \, \, \lra$'').  
Accordingly, Theorem~\ref{thm:i=1-r=1} generalizes this 1-species result: a network with one reversible reaction and one irreversible reaction is multistationary if and only if (1) it is consistent (which we recall is required to admit even one positive steady state) and (2) one of the embedded networks obtained by removing all but one species has the form  
``$\lra \, \, \to$''
or
 ``$ \leftarrow \, \, \lra$''.

For instance, the network $G=\{B \lra A+2B,~3A+B \to 2A\}$ is consistent, but removing 1 species yields $\{0 \lra A,~ 2A \leftarrow 3A \}$ or $\{0 \leftarrow B \lra 2B \}$, neither of which has the form ``$\lra \, \, \to$'' or ``$ \leftarrow \, \, \lra$'' (these conditions require the 3 reactants to be distinct).  Thus, $G$ is not multistationary.

Much of Theorem~\ref{thm:i=1-r=1} is due to Boros~\cite{revisiting}. Our contribution is to reinterpret Boros's multistationarity criterion via the signs of the $\beta_i$'s and via the geometry of the box diagram. 
\begin{theorem}[Classification of multistationary networks with one reversible reaction and one irreversible reaction] \label{thm:i=1-r=1}
Let $G$ be a network that consists of one reversible-reaction pair $y \lra y'$
and one irreversible reaction $\widetilde y \to \widetilde y'$.  
For each species $i=1,2,\dots,s$, let $\beta_i:=\beta_i(y \to y',~\widetilde y \to \widetilde y')$.  
Then the following are equivalent:
\begin{enumerate}[(1)]
\item $G$ is multistationary (i.e., $\capPSS(G) \geq 2$). 
\item 
$y'-y =- \lambda (\widetilde y' - \widetilde y)$ for some $0 \neq \lambda \in \mathbb{R}$, and
after switching the roles of $y$ and $y'$ if necessary so that $\lambda>0$, 
the inequality $\beta_i<0$ holds for some species $i$. 
\item 
 $y'-y =- \lambda (\widetilde y' - \widetilde y)$ for some $0 \neq \lambda \in \mathbb{R}$,  and
$\max\{y_i,y_i'\}< \widetilde y_i < \widetilde y_i'$ or 
$\widetilde y_i' <
\widetilde y_i <
\min\{y_i,y_i'\} $ 
holds for some species $i$.
\item  $y'-y =- \lambda (\widetilde y' - \widetilde y)$ for some $0 \neq \lambda \in \mathbb{R}$, and, for some species $i$, the embedded network of $G$ obtained by removing all species except $i$ is a 2-alternating network (``$\lra \, \, \to$''
or
 ``$ \leftarrow \, \, \lra$'').%
\end{enumerate}

\end{theorem}
\begin{proof}

By Lemma~\ref{lem:consistent}, a necessary condition for multistationarity is that the reaction vectors $y'-y$, $y-y'$, and $\widetilde y'-\wt y$ are positively linearly dependent.  Therefore, $\widetilde y'- \widetilde y$ must be a (nontrivial) scalar multiple of $y'-y$.  Under this assumption,  
Boros proved that when the four complexes are distinct (i.e., the network has two linkage classes), the network is multistationary if and only if 
$\min\{y_i,y_i'\}<
\max\{y_i,y_i'\}< \widetilde y_i < \widetilde y_i'$ or 
$\widetilde y_i' <
\widetilde y_i <
\min\{y_i,y_i'\} < \max\{y_i,y_i'\}$  holds for some $i$~\cite{revisiting}.  These inequalities are precisely the
the max/min conditions in condition (3), once we note that $\min\{y_i,y_i'\}<
\max\{y_i,y_i'\}$ if and only if $\wt y_i \neq \wt y_i'$, and analogously $\min\{\wt y_i, \wt y_i'\}<
\max\{\wt y_i, \wt y_i'\}$ if and only if $ y_i \neq  y_i'$.  In other words, (1)$\Leftrightarrow$(3) in the two-linkage-class case.  

To complete the proof of (1)$\Leftrightarrow$(3), we claim that if $G$ has only one linkage class (and comprises one reversible-reaction pair and one irreversible reaction), then $G$ is not multistationary.  
Indeed, such a network has deficiency zero or one, and has a unique terminal strong linkage class, so it is not multistationary, by the deficiency zero or deficiency one theorem (Theorems~\ref{lem:thm:def-0} and~\ref{lem:thm:def-1}).

Now we prove that (2)$\Leftrightarrow$(3).  Assume that the reaction vectors are scalar multiples of each other.  Furthermore, assume that (after switching $y$ and $y'$ if necessary) $y'-y =- \lambda (\widetilde y' - \widetilde y)$, where $\lambda >0$.  (Such a switch does {\em not} affect the inequalities in condition~(3).)  Then, 
the inequality $\beta_i=(y_i' - y_i)(\wt y_i - y_i)<0$ holds 
if and only if 
$y_i' < y_i < \widetilde y_i$ or $\wt y_i < y_i <y_i'$, 
which in turn holds if and only if 
$ y_i' < y_i < \widetilde y_i < \wt y_i'$ or $\wt y_i' < \wt y_i < y_i <y_i'$ 
(when $\lambda>0$), 
and these are equivalent to the inequalities in condition~(3) when $\lambda >0$. 
Hence, (2)$\Leftrightarrow$(3). 

Finally, the equivalence (3)$\Leftrightarrow$(4) holds: $\min\{y_i,y_i'\}<
\max\{y_i,y_i'\}< \widetilde y_i < \widetilde y_i'$ corresponds to ``$\lra \, \, \to$'', 
and 
$\max\{y_i,y_i'\}> 
\min\{y_i,y_i'\} >\widetilde y_i > \widetilde y_i'$ 
corresponds to 
 ``$ \leftarrow \, \, \lra$''.%
\end{proof}

\begin{remark} \label{rmk:multistab-i=1-r=1}
We wish to extend Theorem~\ref{thm:i=1-r=1} to include a criterion for multistability.  However, at present, we do not know precisely which networks with 1 irreversible and 1 reversible reaction are multistable.  In the 1-species case, none are multistable (Theorem~\ref{thm:s=1-r=1}).  For the 2-species case, some are and some are not: $\{ A \lra B,~2A+B \to 3A\}$ is bistable, but the related network $\{ A \to B,~2A+B \lra 3A\}$ is not (it admits 1 stable and 1 unstable positive steady state).  
\end{remark}

 The following result is analogous to Corollaries~\ref{cor:at-most-bi-s=1-r=1} and~\ref{cor:at-most-bi-r=2} above:
\begin{corollary} \label{cor:at-most-bi-i=1-r=1}
Suppose that $G$ is an at-most-bimolecular reaction network. If $G$ has exactly one reversible and one irreversible reaction, then $G$ is {\em not} nondegenerately multistationary. 
 \end{corollary}
 \begin{proof}
A multistationary network with one reversible and one irreversible reaction must satisfy condition (3) of Theorem~\ref{thm:i=1-r=1}.  Thus, for some species $i$, either
$\widetilde y_i' \geq 3$ or $\max\{ y_i, y_i' \} \geq 3$.  Hence, at least one of the complexes $\widetilde y'$, $y$, and $y'$ is {\em not} at-most-bimolecular.
 \end{proof}

Theorem~\ref{thm:i=1-r=1} yields {\em no} new embedded-minimal multistationary networks:
\begin{corollary} \label{cor:i=1-r=1-atoms} 
A reaction network $G$ that consists of exactly one reversible and one irreversible reaction is an embedding-minimal multistationary network if and only if
$G$ is a 2-alternating network (equivalently, a 1-species network of the form ``$\lra \to$'' or ``$\leftarrow \lra$'').
 \end{corollary}

 \begin{proof}
 The implication $\Leftarrow$ is an earlier result (Corollary~\ref{cor:s=1-r=1-atoms}).  For $\Rightarrow$, Theorem~\ref{thm:i=1-r=1}, specifically part (4), states that $G$ must contain a 2-alternating network as an embedded network.
\end{proof}

\subsection{Networks with two reversible reactions} \label{subsec:2-rev}

We saw in the previous subsection that Theorem~\ref{thm:i=1-r=1} generalizes the corresponding result for 1-species networks with 1 reversible and 1 irreversible reaction.  Next we will see that the same is true for networks with 2 reversible reactions (Theorem~\ref{thm:rev=2} below).  Namely, the corresponding 1-species networks must be a 3-alternating network, which means here that it has the form ``$\lra \, \, \lra$'', e.g., $A\lra 2A,~3A \lra 4A$).   Theorem~\ref{thm:rev=2} generalizes this 1-species result: a network with two reversible reactions is multistationary if and only if (1) all the reaction vectors are scalar multiples of each other (so, the system is 1-dimensional), and (2) one of the embedded networks obtained by removing all but one species has the form  
``$\lra \, \, \lra$''.

For instance, the network $\{ 3A \lra 2A+B,~ A+2B \lra 3B \}$ is multistationary: up to scaling, all the reaction vectors are $(1,-1)$, and the embedded network obtained by removing $A$ is $\{ 0 \lra B,~2B \lra 3B\}$, which has the form
``$\lra \, \, \lra$''.
Among multistationary networks that are mass-preserving, this network is one with the fewest numbers of species and complexes~\cite[\S 5]{Smallest}.

Like Theorem~\ref{thm:i=1-r=1}, much of Theorem~\ref{thm:rev=2} is due to Boros~\cite{revisiting}. 
\begin{theorem}[Classification of multistationary networks with two reversible reactions] \label{thm:rev=2}
Let $G$ be a network that consists of two reversible-reaction pairs, denoted by $y \lra y'$ and $\widetilde y \lra \widetilde y'$.  
Then the following are equivalent:
\begin{enumerate}[(1)]
\item $G$ is multistationary (i.e., $\capPSS(G) \geq 2$).
\item $y'-y =- \lambda (\widetilde y' - \widetilde y)$ for some $0 \neq \lambda \in \mathbb{R}$, and, for some species $i$, either 
 	\begin{align*}
	& \min \{y_i,y_i'\}< \max \{y_i,y_i'\}<\min \{\widetilde y_i, \widetilde y_i'\}<\max \{\widetilde y_i, \widetilde y_i'\} \quad  {\rm or} \\
	& \min  \{\widetilde y_i, \widetilde y_i'\}<\max  \{\widetilde y_i, \widetilde y_i'\}<\min \{ y_i,  y_i'\} <\max \{ y_i,  y_i'\}~.
	\end{align*}
\item  $y'-y =- \lambda (\widetilde y' - \widetilde y)$ for some $0 \neq \lambda \in \mathbb{R}$, and, for some species $i$, the embedded network of $G$ obtained by removing all species except $i$ is a 3-alternating network (``$\lra \, \, \lra$'').%
\end{enumerate}
\end{theorem}

\begin{proof}

We claim that if $G$ has only one linkage class, then it is not multistationary.  Indeed, such a network is weakly reversible and has deficiency zero or one, so by the deficiency zero or deficiency one theorem (Theorems~\ref{lem:thm:def-0} and~\ref{lem:thm:def-1}), respectively, $G$ is {not} multistationary.  Also, for networks with one linkage class, then for all species~$i$, one of the two inequalities in condition (2) holds.  Hence, we have proven (1)$\Leftrightarrow$(2) in the one-linkage-class case.

Now consider the two-linkage-class case (i.e., when the four complexes are distinct).  If the two reaction vectors $y'-y$ and $\widetilde y' - \widetilde y$ are {\em not} linearly dependent, then the network is weakly reversible with deficiency zero, so is not multistationary by the deficiency zero theorem.  
When the two reaction vectors {\em are} linearly dependent, 
Boros proved that the network is multistationary if and only if 
 condition (2) holds~\cite{revisiting}. Hence, (1)$\Leftrightarrow$(2) in the two-linkage-class case.   

It remains only to prove (2)$\Leftrightarrow$(3).  Indeed, the embedded network obtained by removing all species except $i$ is 3-alternating if and only if the four complexes (when projected to $i$) are distinct and both (projected) complexes in one of the reversible-reaction pairs are strictly less than those of the other pair.  This condition is captured precisely by the inequalities in (2).
\end{proof}

\begin{remark} \label{rmk:multistab-rev=2}
As in Remarks~\ref{rmk:multistab} and~\ref{rmk:multistab-i=1-r=1}, we would like to extend Theorem~\ref{thm:rev=2} to include a criterion for multistability.
\end{remark}

The proof of the following result is essentially the same as that for Corollary~\ref{cor:at-most-bi-i=1-r=1} earlier:
\begin{corollary} \label{cor:at-most-bi-rev=2}
Suppose that $G$ is an at-most-bimolecular reaction network. If $G$ consists of two reversible reactions, then $G$ is {\em not} nondegenerately multistationary. 
 \end{corollary}

\begin{example} \label{ex:concordant}
The network $\{A+B \lra C,~2A \lra B\}$ is at-most-bimolecular, so is not multistationary (Corollary~\ref{cor:at-most-bi-rev=2}).  This also follows from Theorem~\ref{thm:rev=2}: the reaction vectors are not scalar multiples of each other.  Moreover, the deficiency zero theorem guarantees a unique positive steady state in each stoichiometric compatibility class; in fact this holds even for general kinetics~\cite[Example~1]{BP}.
\end{example}

Theorem~\ref{thm:rev=2} yields {\em no} embedded-minimal multistationary networks:
\begin{corollary} \label{cor:rev=2-atoms} 
Every reaction network $G$ that consists of two pairs of reversible reactions is {\em not} an embedding-minimal multistationary or multistable network.
 \end{corollary}

 \begin{proof}
Some embedded $1$-species network of $G$ has the form
 ``$\lra \lra$'' (Theorem~\ref{thm:rev=2}). This gives a further embedded network of the form ``$\leftarrow \lra$'', which is nondegenerately multistationary (Corollary~\ref{cor:s=1-r=1-atoms}).
 \end{proof}

\section{Discussion} \label{sec:openQ}
Our work was motivated by Question~\ref{q:main}: which reaction networks are multistationary?  We do not expect an easy answer to this question in general, as there is likely no finite criterion for multistationarity.  Nevertheless, we succeeded in answering this question for networks with 1 species or up to 2 reactions, each possibly reversible. 

Our results point the way forward for analyzing larger networks: the geometry of reaction diagrams will play a key role, as will sign conditions on the $\beta_i(y\to y', \wt y \to \wt y')$'s. 
Indeed, although the multistationary networks analyzed here may not appear often in applications (none are at-most-bimolecular), many of the techniques we used to assess multistationarity apply to more general networks.

Similar techniques may help answer a few questions raised by our work.  First, among those networks we analyzed for which the stoichiometric subspace is 1-dimensional, the multistationary ones are precisely the ones with an
embedded 1-species network with arrow diagram $(\la, \to)$ and another with  arrow diagram $(\to, \la)$.
We ask now whether this holds more generally:
	\begin{question} \label{q:1-d} 
Consider a network $G$ with 1-dimensional stoichiometric subspace.
For $G$ to be multistationary, is it necessary for $G$ to have an 
embedded 1-species network 
with arrow diagram $(\la, \to)$ and another with arrow diagram $(\to, \la)$?
Is it sufficient?
	\end{question}

We also recall questions posed in Remarks~\ref{rmk:multistab},~\ref{rmk:multistab-i=1-r=1}, and~\ref{rmk:multistab-rev=2}:
	\begin{question} \label{q:multistable} 
	Which networks that consist of (a) 2 irreversible reactions, (b) 1 irreversible and 1 reversible reaction, or (c) 2 reversible reactions, are multistable?
	\end{question}

Moreover, these questions point to a larger gap in the theory: methods exist for assessing multistationarity, but few can establish {\em multistability} or even nondegenerate multistationarity.  Our results form a step in this direction: the classification of small nondegenerately multistationary networks (and many of the multistable ones)
hints that a careful analysis of reaction diagrams may yield a general criterion (necessary or sufficient conditions) for bistability.



\subsection*{Acknowledgements}
The authors thank Timo DeWolff, Alicia Dickenstein, Magal\'{i} Giaroli, and J.~Maurice Rojas for helpful discussions.  The authors also acknowledge two conscientious referees whose insightful comments improved our work.  
AS was supported by the NSF (DMS-1312473).



\bibliographystyle{plain}
\bibliography{multistationarity}

\end{document}